\documentclass[envcountsame
]{llncs}

\usepackage{amsmath,amssymb}
\usepackage{latexsym}
\usepackage{amsfonts} 
\usepackage[dvipdfm]{color,graphicx}
\usepackage[english]{babel}
\usepackage{ID_1}

\newcounter{fenumerate}

\newenvironment{fenumerate}
{
  \begin{list}{($f$.\arabic{fenumerate})}{\usecounter{fenumerate}}
}
{
  \end{list}
}

\title{A Simplified Characterisation of Provably Computable Functions of the
System $\mathbf{ID}_1$ of 
Inductive Definitions \\ (Technical Report)}

\author{Naohi Eguchi\inst{1}\thanks{The first author is supported by the research project \emph{Philosophical Frontiers in
Reverse Mathematics} sponsored by the John Templeton Foundation.}
 and Andreas Weiermann\inst{2}}
\institute{%
Mathematical Institute,
Tohoku University,
Japan \\
\email{eguchi@math.tohoku.ac.jp}
\and 
Department of Mathematics, 
Ghent University,
Belgium \\
\email{weiermann@cage.ugent.be}
}

\date{May, 2012}

\begin{document}

\maketitle

\begin{abstract}
We present a simplified and streamlined 
 characterisation of provably total computable functions of the theory $\mathbf{ID}_1$ of
 non-iterated inductive definitions.
The idea of the simplification is to employ the method of
 operator-controlled derivations that was originally introduced by
 Wilfried Buchholz and afterwards applied by the second author to a
 characterisation of provably total computable functions of Peano
 arithmetic $\mathrm{PA}$.
\\

\end{abstract}

\section{Introduction}

As stated by G\"odel's second incompleteness theorem, any reasonable
consistent formal system has an unprovable $\Pi^0_2$-sentence 
that is true in the standard model of arithmetic.
This means that the total (computable) functions whose totality is
provable in a consistent system, which are known as {\em provably computable
functions} or {\em provably total functions}, form a proper subclass of
total computable
functions. 
It is natural to ask how we can describe the provably total functions of
a given system.
Not surprisingly provably (total) computable functions are closely related to provable
well-ordering, i.e., {\em ordinal analysis}.
Up to date ordinal analysis for quite strong systems has been
accomplished by M. Rathjen \cite{Rath91,Rath94} or T. Arai
\cite{arai03,arai04}.
On the other hand several successful applications of techniques from
ordinal analysis to characterisations of provably computable functions have
been provided by B. Blankertz and A. Weiermann \cite{BW96}, W. Buchholz \cite{buch01},
Buchholz, E. A. Cichon and Weiermann \cite{BCW94},
M. Michelbrink \cite{Mich06}, or G. Takeuti \cite{takeuti87}.
Surveys on characterisations of provably computable functions of fragments of
Peano arithmetic $\mathrm{PA}$
contain the monograph \cite{FW98} by M. Fairtlough and S. S. Wainer.

Modern ordinal analysis is based on the method of {\em local
predicativity}, that was first introduced by W. Pohlers,
c.f. \cite{Poh89,Poh98}.
Successful applications of local predicativity to provably computable functions
contain works by Blankertz and Weiermann
\cite{weier96} and by Weiermann \cite{weier99}.
However, to the authors' knowledge, the most successful way in ordinal
analysis is based on the method of
{\em operator-controlled derivations}, an essential
simplification of local predicativity, that was introduced by Buchholz
\cite{Buch92}.
In \cite{weier06} the second author successfully applied the method of
operator-controlled derivations to a streamlined characterisation of
provably computable functions of 
$\mathrm{PA}$.
(See also  \cite[Section 2.1.5]{Poh98}.)
Technically this work aims to lift up the characterisation in
\cite{weier06} to an impredicative system $\mathbf{ID}_1$ of
non-iterated inductive definitions.
We introduce an ordinal notation system $\OTO$ and define a computable function 
$f^\alpha$ for a starting number-theoretic function 
$f: \mathbb{N} \rightarrow \mathbb{N}$
by transfinite recursion on $\alpha \in \OTO$.
The ordinal notation system $\OTO$ comes from a draft
\cite{weier_draft} of the second author and the transfinite
definition of $f^\alpha$ comes from \cite{weier06}.
We show that a function is provably computable in $\mathbf{ID}_1$ if and
only if it is a Kalmar elementary function in
$\{ \suc^\alpha \mid \alpha \in \OTO \text{ and } \alpha < \Omega \}$,
where $\suc$ denotes the successor function 
$m \mapsto m+1$ and $\Omega$ denotes the least non-recursive ordinal.
(Corollary \ref{c:main})

\section{Preliminaries}

In order to make our contribution precise, in this preliminary section
we collect the central notions.
We write $\mathcal{L}_{\mathrm{PA}}$ to denote the standard language of first order
theories of arithmetic.
In particular we suppose that the constant $0$ and the successor
function symbol $S$ are included in $\mathcal{L}_{\mathrm{PA}}$.
For each natural $m$ we use the notation $\ul{m}$ to denote the
corresponding numeral built from $0$ and $S$.
Let a set variable $X$ denote a subset of $\mathbb{N}$. 
We write $X(t)$ instead of $t \in X$ and $\mathcal{L}_{\mathrm{PA}}(X)$ for 
$\mathcal{L}_{\mathrm{PA}} \cup \{ X \}$.
Let $\mathsf{FV}_1 (A)$ denote the set of free number variables appearing
in a formula $A$ and $\FV_2 (A)$ the set of free set variables in $A$.
And then let $\FV (A) := \FV_1 (A) \cup \FV_2 (A)$.
For a fresh set variable $X$ we call an an 
$\mathcal{L}_{\mathrm{PA}} (X)$-formula $\mathcal{A} (x)$ a
{\em positive operator form} if
$\FV_1 (\mathcal A (x)) \subseteq \{ x \}$,
$\FV_2 (\ofA (x)) = \{ X \}$, and
$X$ occurs only positively in $\mathcal A$.

Let $\mathsf{FV}_1 (\mathcal{A} (x)) = \{ x \}$.
For a formula $F(x)$ such that $x \in \mathsf{FV}_1 (F(x))$ we write 
$\mathcal A (F, t)$ to denote the result of replacing in 
$\mathcal{A} (t)$ every subformula $X(s)$ by $F(s)$.
The language $\LID$ of the {\em theory $\mathbf{ID}_1$ of
non-iterated inductive definitions} is defined by
$\LID := \mathcal L_{\mathrm{PA}} \cup \{ P_{\mathcal A} \mid \mathcal A 
                      \text{ is a positive operator form} \}$ 
where for each positive operator form $\mathcal{A}$, 
$P_{\mathcal A}$ denotes a new unary predicate symbol.
We write $\mathcal{T} (\LID, \mathcal{V})$ to denote the set of 
$\LID$-terms and $\mathcal{T} (\LID)$ to denote the set of closed 
$\LID$-terms.
The axioms of $\mathbf{ID}_1$ consist of the axioms of Peano arithmetic 
$\mathrm{PA}$ in
the language $\LID$ and the following new axiom schemata $(\mathsf{ID}_1)$ and
$(\mathsf{ID}_2)$:

\begin{description}
\item[$(\mathsf{ID}1)$] 
$\forall x (\mathcal A (P_{\mathcal A}, x) \rightarrow     
            P_{\mathcal A} (x))$.
\item[$(\mathsf{ID}2)$] (The universal closure of) 
$\forall x (\mathcal A (F, x) \rightarrow F(x)
         ) \rightarrow 
 \forall x (P_{\mathcal A} (x) \rightarrow F(x)
         )$,
where $F$ is an $\LID$-formula.
\end{description}

For each $n \in \mathbb{N}$ we write $\mathrm{I \Sigma}_n$ to denote the
fragment of Peano arithmetic $\mathrm{PA}$
with induction restricted to $\Sigma^0_n$-formulas.
Let $k$ be a natural number and $f: \mathbb{N}^k \rightarrow \mathbb{N}$ a
number-theoretic function and $T$ be a theory of arithmetic containing 
$\mathrm{I \Sigma}_1$.
Then we say $f$ is {\em provably computable in $T$} or 
{\em provably total in $T$} if there exists a
$\Sigma^0_1$-formula $A_f (x_1, \dots, x_k, y)$ such that the following hold:
\begin{enumerate}
\item $\FV (A_f) = \FV_1 (A_f) = \{ x_1, \dots, x_k, y \}$.
\item For all $\vec{m}, n \in \mathbb{N}$, 
      $f(\vec{m}) = n$ holds if and only if 
      $A_f (\ul{\vec{m}}, \ul{n})$ is true in the standard model
      $\mathbb{N}$ of $\mathrm{PA}$.
\item $\forall \vec{x} \exists ! y A_f (\vec{x}, y)$ is a theorem in
      $T$.
\end{enumerate}

It is well known that the provably computable functions of the theory
$\mathrm{I \Sigma}_1$ coincide with the primitive recursive functions.
It is also known that the provably computable functions of the theory
$\mathrm{I \Sigma}_2$ coincide with the P\'eter's multiply recursive functions.

\section{A non-recursive ordinal notation system $\OT$}

In this section we introduce a {\em non-recursive} ordinal notation system 
$\OT = \langle \OT, < \rangle$.
This new ordinal notation system is employed in the next section.
For an element $\alpha \in \OT$ let
$\OT \seg \alpha$ denote the set
$\{ \beta \in \OT \mid \beta < \alpha \}$. 

\begin{definition}
We define three sets 
$\SC \subseteq \AI \subseteq \OT$ of ordinal terms and a set $\opF$ of unary
 function symbols simultaneously. 
Let $0$, $\varphi$, $\Omega$, $\Suc$, $\EN$ and $+$ be distinct symbols.
\begin{enumerate}
\item $0 \in \OT$ and $\Omega \in \SC$.
\item $\{ \Suc, \EN \} \subseteq \opF$.
\item If $\alpha \in \OT \seg \Omega$, then 
      $\Suc (\alpha) \in \OT$ and 
      $\EN (\alpha) \in \AI$.
\item If $\{ \alpha_1, \dots, \alpha_l \} \subseteq \AI$ and
      $\alpha_1 \geq \cdots \geq \alpha_l$, then 
      $\alpha_1 + \cdots + \alpha_l \in \OT$.
\item If $\{ \alpha, \beta \} \subseteq \OT \seg \Omega$, then 
      $\varphi \alpha \beta \in \AI$.
\item If $\alpha \in \OT$ and $\xi \in \OT \seg \Omega$, then
      $\Omega^\alpha \cdot \xi \in \AI$.
\item If $F \in \opF$, $\alpha \in \OT$ and $\xi \in \OT \seg \Omega$,
      then $F^\alpha (\xi) \in \SC$.
\item If $F \in \opF$ and $\alpha \in \OT$, then
      $F^\alpha \in \opF$.
\end{enumerate}
\end{definition}

By definition $F(\xi) \in \OT$ holds if $F^\alpha (\xi) \in \OT$ for
some $\alpha \in \OT$.
We write $\omega^\alpha$ to denote $\varphi 0 \alpha$ and 
$m$ to denote 
$\omega^0 \cdot m = 
 \underbrace{\omega^0 + \cdots + \omega^0}_{m \text{ many}}$.

Let $\Ord$ denote the class of ordinals and $\Lim$ the class of limit ones.
We define a semantic $[\cdot]$ for $\OT$, i.e.,
$[\cdot]: \OT \rightarrow \Ord$.
The well ordering $<$ on $\OT$ is defined by
$\alpha < \beta \Leftrightarrow [\alpha] < [\beta]$.
Let $\Omega_1$ denote the
least non-recursive ordinal $\oCK$. 
For an ordinal $\alpha$ we write
$\alpha =_{NF} \Omega_1^{\alpha_1} \cdot \beta_1 + \cdots +
               \Omega_1^{\alpha_l} \cdot \beta_l$ 
if 
$\alpha > \alpha_1 > \dots > \alpha_l$,
$\{ \beta_1, \dots, \beta_l \} \subseteq \Omega_1$, and 
$\alpha = \Omega_1^{\alpha_1} \cdot \beta_1 + \cdots +
          \Omega_1^{\alpha_l} \cdot \beta_l$. 
Let $\varepsilon_\alpha$ denote the $\alpha$th epsilon number.
One can observe that for each ordinal
$\alpha < \varepsilon_{\Omega_1 + 1}$ there uniquely exists a set
$\{ \alpha_1, \dots, \alpha_l, \beta_1, \dots, \beta_l \}$ of  ordinals such
that
$\alpha =_{NF} \Omega_1^{\alpha_1} \cdot \beta_1 + \cdots +
               \Omega_1^{\alpha_l} \cdot \beta_l$. 
For a set $K \subseteq \Ord$ and for an ordinal $\alpha$ we will write 
$K < \alpha$ to abbreviate $(\forall \xi \in K) \xi < \alpha$, and
dually $\alpha \leq K$ to abbreviate
$(\exists \xi \in K) \alpha \leq \xi$.

\begin{definition}[Collapsing operators]
\begin{enumerate}
\item Let $\alpha$ be an ordinal such that
$\alpha =_{NF} \Omega_1^{\alpha_1} \cdot \beta_1 + \cdots +
               \Omega_1^{\alpha_l} \cdot \beta_l 
 < \varepsilon_{\Omega_1 +1}$.
The set $K_{\Omega} \alpha$ of \emph{coefficients} of 
$\alpha$ is defined by
\[
 K_{\Omega} \alpha = \{ \beta_1, \dots, \beta_l \} \cup
 K_{\Omega} \alpha_1 \cup \cdots \cup K_{\Omega} \alpha_l. 
\]
\item Let $F: \Ord \rightarrow \Ord$ be an ordinal function.
      Then a function $F^\alpha: \Ord \rightarrow \Ord$ is
      defined by transfinite recursion on $\alpha \in \Ord$ by 
\[
      \left\{
       \begin{array}{rcl}
       F^0 (\xi) &=& F (\xi), \\
       F^\alpha (\xi) &=&  \min
       \{ \gamma \in \Ord \mid \omega^\gamma = \gamma, \
          K_{\Omega} \alpha \cup \{ \xi \} < \gamma \text{ and } \\
       && \hspace{2.5cm}
          (\forall \eta < \gamma) (\forall \beta < \alpha) 
          (K_{\Omega} \beta < \gamma \Rightarrow 
           F^\beta (\eta) < \gamma
          )
       \}.
       \end{array}\right.
\]
\end{enumerate}
\end{definition}

\begin{corollary}
Let $F: \Ord \rightarrow \Ord$ be an ordinal function.
Then
$F^\beta (\eta) < F^\alpha (\xi)$ holds if one of the following holds.
\begin{enumerate}
\item $\beta < \alpha$ and 
      $K_{\Omega} \beta \cup \{ \eta \} < F^\alpha (\xi)$. 
\item $\alpha \leq \beta$ and 
      $F^\beta (\eta) \leq K_{\Omega} \alpha$.
\end{enumerate}
\end{corollary}

\begin{proposition}
Suppose that $\alpha < \varepsilon_{\Omega_1 +1}$, a function
$F: \Ord \rightarrow \Ord$ has a $\Sigma_1$-definition in the
 $\Omega_1$-th stage $L_{\Omega_1}$ of the constructible hierarchy 
$(L_\alpha)_{\alpha \in \Ord}$ and that
$F(\xi) < \Omega_1$ for all $\xi < \Omega_1$.
Then $F^\alpha$ also has a $\Sigma_1$-definition in $L_{\Omega_1}$ and 
$F^\alpha (\xi) < \Omega_1$ holds for all $\xi < \Omega_1$.
\end{proposition}

\begin{proof}
By induction on $\alpha < \varepsilon_{\Omega_1 +1}$.
If $\alpha =0$, then $F^0$ a $\Sigma_1$-function since so is $F$, and 
$F^0 (\xi) = F(\xi) < \Omega_1$ for all $\xi < \Omega_1$.
Suppose $\alpha > 0$.
From elementary facts in generalised recursion theory, c.f. Barwise's
 book \cite{Bar75}, careful readers will observe that $F^\alpha$ has a
 $\Sigma_1$-definition in $L_{\Omega_1}$ since 
``$\xi \in K_{\Omega} \alpha$'' can be expressed by a $\Delta_0$-formula.
To see that $F^\alpha (\xi)$ for all $\xi < \Omega_1$ let us define a function $\psi: \omega \rightarrow \varepsilon_{\Omega_1}$ by
\begin{eqnarray*}
\psi (0) &=& \min \{ \gamma < \varepsilon_{\Omega_1 +1} \mid
                     \omega^\gamma = \gamma \text{ and }
                     K_{\Omega} \alpha \cup \{ \xi \} < \gamma
                  \}, \\
\psi (m+1) &=& \min \{ \gamma < \varepsilon_{\Omega_1 +1} \mid
                     \omega^\gamma = \gamma, \
                     K_{\Omega} \alpha \cup \{ \xi \} < \gamma
                     \text{ and } \\
&& \hspace{1.5cm}
                     (\forall \eta < \psi (m))(\forall \beta < \alpha)
                     [K_{\Omega} \beta < \psi (m) \Rightarrow
                      F^\beta (\eta) < \gamma]              
                    \}.
\end{eqnarray*}
We can see that $\psi$ is a $\Sigma_1$-function in the same way as we
 see that $F^\alpha$ is so.

\begin{claim}
$\psi (m) < \Omega_1$ for all $m \in \omega$.
\end{claim}

We show that $\psi (m) < \Omega_1$ holds by (side) induction on
 $m$.
In the base case, $\psi (0) < \Omega_1$ holds since 
$K_{\Omega} \alpha \cup \{\xi\} < \Omega_1$ and 
$\Omega_1$ is closed under the
 function $[\EN]$.
Consider the induction step.
Let $\eta < \psi (m)$.
Then Side Induction Hypothesis implies $\eta < \psi (m) < \Omega_1$.
Hence (Main) Induction Hypothesis enables us to deduce 
$F^\beta (\eta) < \Omega_1$ for all $\beta < \alpha$.
Let us define a function 
$G: \{ \beta < \alpha \mid K_{\Omega} \beta < \psi (m) \}
 \rightarrow \Omega_1$
by $\beta \mapsto F^\beta (\eta)$.
One can see that $G$ is a $\Sigma_1$-function.
On the other hand 
$\# \{ \beta < \alpha \mid K_{\Omega} \beta < \psi (m) \} \leq \omega$
since $\psi (m) < \Omega_1$.
Here we recall that $\Omega_1$ denotes the least recursively regular
 ordinal 
$\omega_1^{\mathrm{CK}}$ and hence $L_{\Omega_1}$ is closed under functions
whose graphs are of $\Sigma_1$ in $L_{\Omega_1}$.
From these we have inequality
$$
 \psi (m+1) \leq
 \sup \{ G (\beta) \mid \beta < \alpha \text{ and }
         K_{\Omega} \beta < \psi (m)
      \}
 < \Omega_1,
$$
concluding the claim.

By the claim
$\psi$ is a $\Sigma_1$-function in 
$L_{\Omega_1}$ from $\omega$ to $\Omega_1$.
Hence $\sup_{m \in \omega} \psi (m) < \Omega_1$.
Define an ordinal $\gamma$ by
$\gamma = \sup_{m \in \omega} \psi (m)$.
Then 
$\omega^\gamma = \gamma$, 
$K_{\Omega} \alpha \cup \{ \xi \} < \gamma$ and 
$K_{\Omega} \beta < \gamma \Rightarrow F^\beta (\eta) < \gamma$
for all $\eta < \xi$ and for all $\beta < \alpha$.
This implies $F^\alpha (\xi) \leq \gamma < \Omega_1$.
\qed
\end{proof}

\begin{proposition}
\label{p:F^a(eta)}
For any $\alpha \in \Ord$, for any $\eta, \xi < \Omega_1$ and for any
 ordinal function $F: \Omega_1 \rightarrow \Omega_1$, if 
$\eta < F^\alpha (\xi)$, then 
$F^\alpha (\eta) \leq F^\alpha (\xi)$.
\end{proposition} 

\begin{proof}
If $\eta \leq \xi$, then $F^\alpha (\eta) \leq F^\alpha (\xi)$ by the
 definition of $F^\alpha (\eta)$.
Let us consider the case $\xi < \eta < F^\alpha (\xi)$.
In this case 
$K_\Omega \alpha \cup \{ \eta \} < F^\alpha (\xi)$
by the definition of $F^\alpha (\xi)$.
Suppose that 
$\beta < \alpha $,
$\gamma < F^\alpha (\xi)$ and $K_\Omega \beta < F^\alpha (\xi)$.
Then 
$F^\beta (\gamma) < F^\alpha (\xi)$
again by the definition of $F^\alpha (\xi)$.
By the minimality of $F^\alpha (\eta)$ we can conclude
$F^\alpha (\eta) \leq F^\alpha (\xi)$.
\qed
\end{proof}

\begin{definition}
\label{def:val}
We define the value $[\alpha] \in \Ord$ of an ordinal term 
$\alpha \in \OT$ by recursion on the length of $\alpha$.
\begin{enumerate}
\item $[0] = 0$ and $[\Omega] = \Omega_1$.
\item $[\alpha + \beta] = [\alpha] + [\beta]$.
\item $\left[ \varphi \alpha \beta \right] = 
       \left[ \varphi \right] [\alpha] [\beta]$, where 
      $\left[ \varphi \right]$ is the standard Veblen function, i.e., \\
      $\left\{
       \begin{array}{rcll}
       [\varphi] 0 \beta &=& \omega^{\beta}, & \\
       \left[ \varphi \right] (\alpha +1) 0 &=&  
       \sup \{ (\left[ \varphi \right] \alpha)^n 0 \mid n \in \omega\},
       & \\
       \left[ \varphi \right] \gamma 0 &=& 
       \sup \{ \left[ \varphi \right] \alpha 0 \mid 
               \alpha < \gamma \} & 
       \text{if } \gamma \in \Lim, \\
       \left[ \varphi \right] (\alpha +1) (\beta +1) &=&
       \sup \{ (\left[ \varphi \right] \alpha)^n ([\varphi] (\alpha +1) \beta +1 
               \mid n \in \omega \}, & \\
       \left[ \varphi \right] \gamma (\beta +1) &=&
       \sup \{ \left[ \varphi \right]  \alpha (\left[ \varphi \right] \gamma \beta +1)
               \mid \alpha < \gamma \} &
       \text{if } \gamma \in \Lim, \\
       \left[ \varphi \right] \alpha \gamma &=&
       \sup \{ \left[ \varphi \right] \alpha \beta \mid 
              \beta < \gamma
            \} & \text{if } \gamma \in \Lim.
       \end{array}\right.
      $
\item $[\Omega^\alpha \cdot \xi] = 
       \Omega_1^{[\alpha]} \cdot [\xi]$.
\item $[\Suc (\alpha)] = [\Suc] ([\alpha])$, where
      $[\Suc]$ denotes the ordinal successor $\alpha \mapsto \alpha +1$.
      Clearly $\{ [\Suc] (\xi) \mid \xi \in \Omega_1 \} \subseteq \Omega_1$.
\item $[\EN (\alpha)] = [\EN] ([\alpha])$, where the function
      $[\EN]: \Ord \rightarrow \Ord$ is defined by
      $[\EN] (\alpha) = \min
       \{ \xi \in \Ord \mid \omega^\xi = \xi \text{ and } \alpha < \xi
      \}$.
      It is also clear that 
      $\{ [\EN] (\xi) \mid \xi \in \Omega_1 \} \subseteq \Omega_1$ holds.
\item $\left[ F^\alpha (\xi) \right] = 
       \left[ F \right]^{[\alpha]} ([\xi])$. 
\label{def:val:4}
\end{enumerate}
\end{definition}


\begin{definition}
For all $\alpha, \beta \in \OT$,
$\alpha < \beta$ if $[\alpha] < [\beta]$, and 
$\alpha = \beta$ if $[\alpha] = [\beta]$.
\end{definition}

We will identify each element $\alpha \in \OT$ with its value $[\alpha] \in \Ord$.
Accordingly we will write $K_{\Omega} \alpha$ instead of 
$K_{\Omega} [\alpha]$ for $\alpha \in \OT$.
Further for a finite set $K \subseteq \Ord$ we write
$K_{\Omega} K$ to denote the finite set 
$\bigcup_{\xi \in K} K_{\Omega} \xi$.
By this identification, $\AI$ is the set of  
{\em additively indecomposable} ordinals and $\SC$ is the set of 
{\em strongly critical} ordinals, i.e,
$\SC \subseteq \AI \subseteq \Lim \cup \{1\} \subseteq \Ord$.

\begin{corollary}
$F^\alpha (\xi) < \Omega$ for any $F \in \opF$ and $\xi < \Omega$.
\label{c:F^a<Omega}
\end{corollary}

\begin{proof}
Proof by induction over the build-up of $F \in \opF$.
\end{proof}


\begin{corollary}
\begin{enumerate}
\item $K_{\Omega} 0 = K_{\Omega} \Omega = \emptyset$.
\item If $K_{\Omega} \alpha < \xi $ and $\xi \in \SC$, then
      $K_{\Omega} \Suc(\alpha) < \xi$. 
\item $K_{\Omega} \EN(\alpha) = \{ \EN(\alpha) \}$ 
      (since $\alpha < \Omega$). 
\item If $K_{\Omega} \alpha \cup K_{\Omega} \beta < \xi$ and
      $\xi \in \SC$, then 
      $K_{\Omega} (\alpha + \beta) < \xi$. 
\item $K_{\Omega} \varphi \alpha \beta = \{ \varphi \alpha \beta \}$
      (since $\alpha, \beta < \Omega$).
      Further, if $\alpha , \beta < \xi$ and $\xi \in \SC$, then
      $\varphi \alpha \beta < \xi$.
\item $K_{\Omega} F^\alpha (\xi) = \{ F^{\alpha} (\xi) \}$
      (since $\xi < \Omega$).
\end{enumerate}
\end{corollary}

By Corollary \ref{c:F^a<Omega} each function symbol from $\opF$ defines a weakly
increasing function
$F: \Omega \rightarrow \Omega$ such that 
$\xi < F(\xi)$ holds for all $\xi \in \Omega$.
In the rest of this section let $F$ denote such a function.
For a finite set $K \subseteq \Ord$ we will use the notation
$F[K] (\xi)$ to abbreviate $F(\max (K \cup \{ \xi \}))$.

\begin{lemma}
\label{lem:F[xi]}
Let $K \subseteq \Ord$ be a finite set such that $K < \Omega$. Then
$(F [K])^\alpha (\xi) \leq F^\alpha [K] (\xi)$ for all
$\xi < \Omega$.
\end{lemma}

\begin{proof}
By induction on $\alpha$.
For the base case
$(F[K])^0 (\xi) = F[K] (\xi) = F^0 [K] (\xi)$.
Suppose $\alpha > 0$.
Then
\begin{equation}
K_{\Omega} \alpha \cup \{ \xi \} < 
F^\alpha (\xi) \leq F^\alpha [K] (\xi).
\label{e:l:F[xi]:1}
\end{equation}
Assume that
$\eta < F^\alpha [K] (\xi)$,
$\beta < \alpha$ and 
$K_{\Omega} \beta < F^\alpha [K] (\xi)$.
Then $\eta < \Omega$, and hence 
$(F[K])^\beta (\eta) \leq F^\beta [K] (\eta)$ by IH.
Hence 
\begin{eqnarray}
(F[K])^\beta (\eta) &\leq& F^\beta [K] (\eta), 
\nonumber \\
&<&
F^\alpha [K] (\eta)
\quad \text{ since }
K_{\Omega} K < F^\alpha [K] (\eta).
\label{e:l:F[xi]:2}
\end{eqnarray}
By conditions (\ref{e:l:F[xi]:1}) and (\ref{e:l:F[xi]:2})
we conclude
$(F[K])^\alpha (\xi) \leq F^\alpha [K] (\xi)$.
\qed
\end{proof}

\begin{lemma}
\label{lem:F^a^b}
$(F^\alpha)^\beta (\xi) \leq F^{\alpha + \beta} (\xi)$ 
for all $\xi < \Omega$.
\end{lemma}

\begin{proof}
By induction on $\beta$.
For the base case 
$(F^\alpha)^0 (\xi) = F^\alpha (\xi) = F^{\alpha + 0} (\xi)$.
Suppose $\beta > 0$.
Then
\begin{equation}
K_{\Omega} \beta \cup \{ \xi \} < F^\beta (\xi) \leq
F^{\alpha + \beta} (\xi).
\label{e:l:F^a^b:1}
\end{equation}
Assume that 
$\eta < F^{\alpha + \beta} (\xi)$,
$\beta ' < \beta$ and 
$K_{\Omega} \beta' < F^{\alpha + \beta} (\xi)$.
Then $\eta < \Omega$, and hence
$(F^\alpha)^{\beta'} (\eta) \leq F^{\alpha + \beta'} (\eta)$ by IH.
Hence
\begin{equation}
(F^\alpha)^{\beta'} (\eta) \leq F^{\alpha + \beta'} (\eta) 
<
F^{\alpha + \beta} (\xi).
\label{e:l:F^a^b:2}
\end{equation}
By conditions (\ref{e:l:F^a^b:1}) and (\ref{e:l:F^a^b:2}) we can
 conclude
$(F^\alpha)^\beta (\xi) \leq F^{\alpha + \beta} (\xi)$.
\qed
\end{proof}

\section{An infinitary proof system $\IDomega$}

This section introduces the main definition of this paper. 
We introduce a new infinitary proof system $\IDomega$ to which the new
ordinal notation system is connected and into which every (finite) proof
in $\mathbf{ID}_1$ can be embedded in good order.
For each positive operator form $\mathcal A$ and for each ordinal term
$\alpha \in (\OT \seg \Omega) \cup \{ \Omega \}$ let $\itP{\ofA}{\alpha}$
be a new unary predicate symbol.
Let us define an infinitary language $\Lstar$ of $\IDomega$ by
$\Lstar = \LPA \cup \{ \neq, \nleq \} \cup
           \{ \itP{\ofA}{\alpha}, \neg \itP{\ofA}{\alpha} \mid 
              \alpha \in (\OT \seg \Omega) \cup \{ \Omega \} \text{ and }
              \ofA \text{ is a positive operator form} \}$.
Let us write $\itP{\ofA}{\Omega}$ to denote $\fpP{\ofA}$ to have the inclusion
$\LID \subseteq \Lstar$.
We write $\mathcal{T} (\Lstar)$ to denote the set of closed 
$\Lstar$-terms.
Specifically, the language $\Lstar$ contains complementary predicate
symbol $\neg P$ for each predicate symbol $P \in \Lstar$.
We note that the negation $\neg$ nor the implication $\rightarrow$ is
not included as a logical symbol.
The negation $\neg A$ is defined via de Morgan's law by
$\neg (\neg P (\vec{t})) :\equiv P (\vec{t})$ for an atomic formula 
$P(\vec{t})$, 
$\neg (A \wedge B) :\equiv \neg A \vee \neg B$,
$\neg (A \vee B) :\equiv \neg A \wedge \neg B$,
$\neg \forall x A :\equiv \exists x \neg A$ and
$\neg \exists x A :\equiv \forall x \neg A$.
The implication $A \rightarrow B$ is defined by
$\neg A \vee B$.
We start with technical definitions.
We will write $\itP{\ofA}{\alpha} t$ and $\neg \itP{\ofA}{\alpha} t$
respectively for 
$\itP{\ofA}{\alpha} (t)$ and $\neg \itP{\ofA}{\alpha} (t)$.

\begin{definition}[Complexity measures of $\Lstar$-formulas]
\begin{enumerate}
\item The {\em length} $\lh (A)$ of an $\Lstar$-formula $A$ is the
      number of the symbols $\itP{\ofA}{\alpha}$, 
      $\neg \itP{\ofA}{\alpha}$, $\vee$, $\wedge$, 
      $\exists$ and $\forall$ occurring in $A$.
\item The {\em rank} $\rk (A)$ of an $\Lstar$-formula $A$.
  \begin{enumerate}
  \item 
      $\rk (\itP{\ofA}{\alpha} t) :=
       \rk (\neg \itP{\ofA}{\alpha} t) := \omega \cdot \alpha$. 
  \item $\rk (A) := 0$ if $A$ is an $\LID$-literal.
  \item $\rk (A \wedge B) := \rk (A \vee B) :=
       \max \{ \rk (A), \rk (B) \} +1$.
  \item $\rk (\forall x A) := \rk (\exists x A) :=
       \rk (A) +1$.
  \end{enumerate}
%
\item The set $\kPi (A)$ of {\em $\Pi$-coefficients} of an $\Lstar$-formula $A$.
  \begin{enumerate}
  \item $\kPi (\itP{\ofA}{\alpha} t) := \{ 0 \}$,
      $\kPi (\neg \itP{\ofA}{\alpha} t) :=
      \{ 0, \alpha \}$.
  \item $\kPi (A) := \{ 0 \}$ if $A$ is an $\LID$-literal.
  \item $\kPi (A \wedge B) := \kPi (A \vee B) :=
       \kPi (A) \cup \kPi (B)$.
  \item $\kPi (\forall x A) := \kPi (\exists x A) :=
       \kPi (A)$.
  \end{enumerate}
%
\item The set $\kSigma (A)$ of {\em $\Sigma$-coefficients} of an
 $\Lstar$-formula $A$. 

$\kSigma (A) := \kPi (\neg A)$. 
\item The set $\kPS (A)$ of all the {\em coefficients} of an
      $\Lstar$-formula $A$. 

$\kPS (A) := \kPi (A) \cup \kSigma (A)$.
\item The set $\kPO (A)$ of {\em $\Pi$-coefficients} of an
      $\Lstar$-formula $A$ less than $\Omega$. 

$\kPO (A) := \kPi (A) \seg \Omega$.

The set $\kSO (A)$ and $\kO (A)$ are defined accordingly.
\end{enumerate}
\end{definition}

By definition $\rk (A) = \rk (\neg A)$, $\kPS (A) = \kPS (\neg A)$ and
$\kO (A) = \kO (\neg A)$.

\begin{definition}[Complexity measures of $\Lstar$-terms]
\begin{enumerate}
\item The {\em value} $\val (t)$ of a term 
$t \in \mathcal{T} (\LID) = \mathcal{T} (\LPA)$ is the value of the
      closed term $t$ in the standard model $\mathbb{N}$ of the Peano
      arithmetic $\mathrm{PA}$.
\item A complexity measure
$\matho: \mathcal{T} (\Lstar) \rightarrow 
 (\OT \seg \Omega) \cup \{ \Omega \}$
is defined by

      $\left\{
       \begin{array}{rcll}
       \matho (t) &:=& 0 & 
       \text{ if } t \in \mathcal{T} (\LID), 
       \\
       \matho (\alpha) &:=& \xi & \text{ if } \alpha \in \OT.
       \end{array}\right.$
\item The \emph{norm} $N(\alpha)$ of $\alpha \in \OT$. 
\begin{enumerate}
\item $N(0) = 0$ and $N (\Omega) =1$.
\item $N (\Suc(\alpha)) = N (\alpha) +1$.
\item $N (\EN (\alpha)) = N (\alpha) +1$.
\item $N(\alpha + \beta) = N(\alpha) + N(\beta)$.
\item $N(\varphi \alpha \beta) = N(\alpha) + N(\beta) +1$,
\item $N (\Omega^\alpha \cdot \xi) = N (\alpha) + N(\xi) + 1$.
\item $N (F^{\alpha} (\xi)) = N (F (\xi)) + N (\alpha)$.
\end{enumerate}
The norm is extended to a complexity measure
$N: \mathcal{T} (\Lstar) \rightarrow \mathbb{N}$ by

      $\left\{
       \begin{array}{rcll}
       N (t) &:=& \val (t) & 
       \text{ if } t \in \mathcal{T} (\LID), 
       \\
       N (\alpha) &:=& N (\alpha) & \text{ if } \alpha \in \OT.
       \end{array}\right.$
\end{enumerate}
\end{definition}

By definition 
$N(\omega^\alpha) = N(\varphi 0 \alpha) = N(\alpha) +1$ and 
$N(m) = N(\omega^0 \cdot m) = m$ for any $m < \omega$.
This seems to be a good point to explain why we contain the constant 
$\Omega$ in $\OT$.
Having that $N(\Omega)=1$ makes some technicality easier.

\begin{definition}
We define a relation $\simeq$ between $\Lstar$-sentences and (infinitary)
propositional $\Lstar$-sentences.
\begin{enumerate}
\item $\neg \itP{\ofA}{\alpha} t: \simeq \AND{\xi \in \OT \seg \alpha}
       \neg \ofA (\itP{\ofA}{\xi}, t)$ and
      $\itP{\ofA}{\alpha} t: \simeq \OR{\xi \in \OT \seg \alpha}
       \ofA (\itP{\ofA}{\xi}, t)$.
\item $A \wedge B: \simeq \AND{\iota \in \{ \ul{0}, \ul{1} \}}
       A_{\iota}$ and
      $A \vee B: \simeq \OR{\iota \in \{ \ul{0}, \ul{1} \}}
       A_{\iota}$ where $A_{\ul{0}} \equiv A$ and
       $A_{\ul{1}} \equiv B$.
\item $\forall x A(x): \simeq
       \AND{t \in \mathcal{T}(\LID)} A(t)$ and
      $\exists x A(x): \simeq
       \OR{t \in \mathcal{T}(\LID)} A(t)$.  
\end{enumerate}
\end{definition}

We call an $\Lstar$-sentence $A$ a {\em $\bigwedge$-type} (conjunctive
type) if 
$A \simeq \bigwedge_{\iota \in J} A_\iota$ for some $A_\iota$, and 
a {\em $\bigvee$-type} (disjunctive type) if 
$A \simeq \bigvee_{\iota \in J} A_\iota$ for some $A_\iota$.
For the sake of simplicity we will write
$\bigwedge_{\xi < \alpha} A_\xi$ instead of
$\bigwedge_{\xi \in \OT \seg \alpha} A_\xi$
and write
$\bigvee_{\xi < \alpha} A_\xi$
accordingly.

\begin{lemma}
\label{lem:kPi}
\begin{enumerate}
\item 
      If either $A \simeq \AND{\iota \in J} A_{\iota}$ or
         $A \simeq \OR{\iota \in J} A_{\iota}$, then
      for all $\iota \in J$,
      $\kPi (A_{\iota}) \subseteq \kPi (A) \cup \{\matho (\iota)\}$ 
      and
      $\kSigma (A_{\iota}) \subseteq \kSigma (A) \cup \{\matho (\iota)\}$.
\label{lem:kPi:1}
\item 
      For any $\alpha \in \OT$,
      if $A \simeq \AND{\xi < \alpha} A_{\xi}$, then
      $(\exists \sigma \in \kPi (A)) (\forall \xi < \alpha)
       [\xi \leq \sigma]$.
\label{lem:kPi:2}
\item For any $\Lstar$-sentence $A$, $\rk (A) = \omega \cdot \max \kPS (A) + n$ for some $n \leq \lh (A)$.
\label{lem:kPi:3}
\item If $\rk (A) = \Omega$, then either 
      $A \equiv \itP{\ofA}{\Omega} t$ or $A \equiv \neg \itP{\ofA}{\Omega} t$.
\label{lem:kPi:4}
\item 
      If either $A \simeq \AND{\iota \in J} A_{\iota}$ or
         $A \simeq \OR{\iota \in J} A_{\iota}$, then
      for all $\iota \in J$,
      $N(\rk (A_\iota)) \leq 
       \max \{ N (\rk (A)), 2 \cdot N(\iota) \}$.
\label{lem:kPi:5}
\end{enumerate}
\end{lemma}


Throughout this section we use the symbol 
$F$ to denote a weakly increasing ordinal function 
$F: \Omega \rightarrow \Omega$
and the symbol $f$ to denote a number-theoretic function
$f: \mathbb{N} \rightarrow \mathbb{N}$ 
that enjoys the following conditions.

\begin{fenumerate}
\item $f$ is a strictly increasing function such that 
      $2m +1\leq f(m)$ for all $m$.
      Hence, in particular,  $n + f(m) \leq f(n+m)$ for all $m$ and $n$.
\label{f:1}
\item $2 \cdot f(m) \leq f(f(m))$ for all $m$.
\label{f:2}
\end{fenumerate}

We will use the notation $f[n](m)$ to abbreviate $f(n+m)$.
It is easy to see that if the conditions ($f$.\ref{f:1}) and 
($f$.\ref{f:2}) hold, then for a fixed $n$ the conditions 
($f[n]$.\ref{f:1}) and ($f[n]$.\ref{f:2}) also hold.


\begin{definition}
Let $f: \mathbb{N} \rightarrow \mathbb{N}$ be a number-theoretic function.
Then a function 
$f^\alpha: \mathbb{N} \rightarrow \mathbb{N}$ is
defined by transfinite recursion on $\alpha \in \OT$ by
\begin{eqnarray*}
f^0 (m) &=& f(m), \\
f^{\alpha} (m) &=& 
 \max \{ f^{\beta} (f^{\beta} (m)) \mid 
         \beta < \alpha \text{ and } 
         \ N (\beta) \leq f[ N (\alpha)] (m)
      \}
\quad \text{if } 0 < \alpha.
\end{eqnarray*}
\end{definition}

\begin{corollary}
\begin{enumerate}
\item If $f$ is strictly increasing, then so
is $f^\alpha$ for any $\alpha \in \OT$.
\item If $\beta < \alpha$ and $N(\beta) \leq f[N(\alpha)](m)$, then
      $f^{\beta} (m) < f^{\alpha} (m)$.
\item $f^\alpha (f^\alpha (m)) \leq f^{\alpha +1} (m)$.
\end{enumerate}
\end{corollary}
 
We note that the function $f^\alpha$ is not a recursive function in
general even if $f$ is recursive since the ordinal notation system 
$\langle \OT, < \rangle$ is not a recursive system.

\begin{example}
The following are examples of $f^\alpha$ in case that 
$\alpha \leq \omega$ and $f$ is the
 successor function $\suc: m \mapsto m+1$.
Let us recall that 
$N(n) = N(\omega^0 \cdot n) = n$ for all $n < \omega$.
\begin{enumerate}
\label{ex:1}
\item $\suc^1 (m) = \suc^0 (\suc^0 (m)) = m+2$.
\item $\suc^2 (m) = \suc^1 (\suc^1 (m)) = m+4$.
\item $\suc^n (m) = m + 2^n$. ($n < \omega$)  
\item 
      $\suc^\omega (m) = m + 2^{m+3}$. 

Let us see that $N(\omega) = 1$ and hence 
$\suc [N(\omega)] (m) = \suc (1 + m) = m+2$.
Hence
$\suc^\omega (m) = f^{m+2} (f^{m+2} (m)) =
 m + 2^{m+2} + 2^{m+2} = m + 2^{m+3}$.
\label{ex:1:4}
\end{enumerate}
\end{example}

\begin{lemma}
\label{lem:N(a)}
Let $\alpha \in \OT$ and $F \in \opF$. 
Then 
$ N (\alpha) \leq f^{F^\alpha (0)} (0)$. 
\end{lemma}

\begin{proof}
By induction over the term-construction of $\alpha \in \OT$.
For the base case 
$N(0) = 0 \leq f(0) \leq f^{F^0(0)} (0)$ and 
$N(\Omega) = 1 \leq f(0) \leq f^{F^\Omega (0)} (0)$.
For the induction step, we only consider the case that
$\alpha = F^{\alpha_0} (\xi)$ for some 
$\alpha_0 \neq 0$ and
for some $\xi < \Omega$.
The remaining cases can be treated in similar ways.
In this case 
$F^{\alpha_0} (0) < F^\alpha (0)$ holds since
$F^{\alpha_0} (0) \leq \{ F^{\alpha_0} (\xi) \}$ 
$= K_\Omega  F^{\alpha_0} (\xi) < F^{F^{\alpha_0} (\xi)} (0) = 
 F^\alpha (0)$.
It is easy to see that
$F^{F(\xi)} (0) < F^\alpha (0)$ holds.
By definition
$N(\alpha) = N(F(\xi)) + N(\alpha_0)$.
By IH
$N(F(\xi)) \leq f^{F^{F(\xi)} (0)} (0)$ and
$N(\alpha_0) \leq f^{F^{\alpha_0} (0)} (0)$.
Hence
\begin{eqnarray*}
N(\alpha) &\leq&
f^{F^{F(\xi)} (0)} (0) + f^{F^{\alpha_0} (0)} (0), \\
&\leq&
f^{F^{\alpha_0} (0)} (f^{F^{F(\xi)} (0)} (0))
\text{ since } 
m + f^{\omega^{\alpha_0}} (0) \leq f^{\omega^{\alpha_0}} (m) 
\text{ for all } m, \\
&\leq&
f^{F^{\alpha_0} (0) + F^{F(\xi)} (0)} 
(f^{F^{\alpha_0} (0) + F^{F(\xi)} (0)} (0)) \\
&\leq& f^{F^\alpha (0)} (0). 
\end{eqnarray*}
To see that the last inequality is true, we can check
$F^{\alpha_0} (0) + F^{F(\xi)} (0) < F^\alpha (0)$ and
$N(F^{\alpha_0} (0) + F^{F(\xi)} (0)) \leq 2 \cdot N(F^\alpha (0)) 
 \leq f[N(F^\alpha (0))] (0)$
from an assumption that $2m \leq f(m)$.
\qed
\end{proof}

\begin{lemma}
\label{lem:(f^a)^b}
Let $\{ \alpha , \beta \} \subseteq \OT \seg \Omega$ and $F \in
 \mathcal{F}$. 
Then, for all $m$,
$(f^\alpha)^\beta (m) \leq f^{F^{\Omega \cdot \alpha + \beta} (0)} (m)$.
\end{lemma}


\begin{proof}
If $\alpha = 0$, then 
$(f^\alpha)^\beta (m) = f^\beta (m) \leq
 f^{F^{\Omega \cdot 0 + \beta} (0)} (m)$.
Suppose $\alpha \neq 0$.
Then we show the assertion by induction on $\beta$.
If $\beta = 0$, then
$(f^\alpha)^\beta (m) = f^\alpha (m) \leq
 f^{F^{\Omega \cdot \alpha} (0)} (m)$.
Suppose $\beta > 0$.
Then there exists $\gamma < \beta$ such that
$N (\gamma) \leq f^\alpha [N (\beta)] (m)$ and
$(f^\alpha)^\beta (m) = (f^\alpha)^\gamma ((f^\alpha)^\gamma (m))$.
By IH
\begin{equation}
(f^\alpha)^\gamma ((f^\alpha)^\gamma (m)) \leq
 f^{F^{\Omega \cdot \alpha + \gamma} (0)} 
 (f^{F^{\Omega \cdot \alpha + \gamma} (0)} (m)).
\label{lem:(f^a)^b:e:1}
\end{equation}
On the other hand 
$N(\beta) \leq f^{F^\beta (0)} (0)$ by Lemma \ref{lem:N(a)}.
Hence
\begin{eqnarray*}
N (\gamma) &\leq& f^\alpha (f^{F^\beta (0)} (m)) 
\quad \text{ since }
m \leq f(m) \leq f^{F^\beta (0)} (m), \\
&\leq&
f^{F^{\Omega \cdot \alpha} (0) + F^\beta (0)} 
  (f^{F^{\Omega \cdot \alpha} (0) + F^\beta (0)} (m)) 
\\
&\leq&
f^{F^{\Omega \cdot \alpha + \beta} (0)} (m).
\end{eqnarray*}
The second inequality holds since
$\{ \alpha, F^\beta (0) \} = 
 K_{\Omega} \alpha \cup \{ F^\beta (0) \} < 
 F^{\Omega \cdot \alpha} (0) + F^{\beta} (0)$.
This implies that
\begin{eqnarray}
N (F^{\Omega \cdot \alpha + \gamma} (0)) 
&\leq&
N(F(0)) + N(\alpha)  + 1 + 
f^{F^{\Omega \cdot \alpha + \beta} (0)} (m) 
\nonumber \\
&\leq&
f [N (F^{\Omega \cdot \alpha + \beta} (0)] 
   (f^{F^{\Omega \cdot \alpha + \beta} (0)} (m)).
\label{lem:(f^a)^b:e:2}
\end{eqnarray}
Further 
$F^{\Omega \cdot \alpha + \gamma} (0) < 
 F^{\Omega \cdot \alpha + \beta} (0)$
holds since 
$K_\Omega \gamma = \{ \gamma \} < \beta$.
This together with the inequality (\ref{lem:(f^a)^b:e:2}) yields that  
\begin{eqnarray*}
(f^\alpha)^\beta (m) &\leq&
 f^{F^{\Omega \cdot \alpha + \gamma} (0)} 
 (f^{F^{\Omega \cdot \alpha + \gamma} (0)} (m))
\quad \text{ by } (\ref{lem:(f^a)^b:e:1}), \\
&\leq&
f^{F^{\Omega \cdot \alpha + \beta} (0)} (m).
\end{eqnarray*}
\qed
\end{proof}

\begin{lemma}
\label{lem:f[n](m)}
\begin{enumerate}
\item $f^\alpha [n] (m) \leq (f [n])^\alpha (m)$.
\label{lem:f[n](m):1}
\item If $n \leq m$, then 
$(f [n])^\alpha (m) \leq f^\alpha [f^\alpha (f(m))](f(m))$.
\label{lem:f[n](m):2}
\end{enumerate}
\end{lemma}

We write $f[n][m]$ to abbreviate $(f[n])(m)$ and 
$f[n]^\alpha$ to abbreviate $(f[n])^\alpha$.

\begin{proof}
{\sc Property} \ref{lem:f[n](m):1}. By induction on $\alpha$.
For the base case
$f^0 [n] (m) = f [n] (m) = f [n]^0 (m)$.
For the induction step, assume $\alpha > 0$. 
Then there exists $\beta < \alpha$ such that
$N(\beta) \leq f[N (\alpha)][n] (m)$ and 
$f^\alpha [n] (m) = f^\beta (f^\beta [n] (m))$.
Hence
\begin{eqnarray*}
f^\alpha [n] (m) &\leq&
f^\beta (f [n]^\beta (m))
\quad \text{ by IH,} \\
&\leq&
f [n]^\beta (f [n]^\beta (m)) \\
&\leq&
f [n]^\alpha (m).
\end{eqnarray*}
The last inequality holds since
$N(\beta) \leq 
 f[N(\alpha)][n](m) = f[n][N(\alpha)] (m)$.

{\sc Property} \ref{lem:f[n](m):2}.
We show that
$f [n]^\alpha (f(m)) \leq f^\alpha [f^\alpha (f(m))](f(m))$
holds for all $m \geq n$
by induction on $\alpha$.
Let $n \leq m$.
For the base case
$f [n]^0 (m) \leq f [n] (m) \leq
 f (m+m) \leq f (f^0 (f(m)) +f(m)) = f^0 [f^0 (f(m))] (f(m))$.
For the induction step, assume $\alpha > 0$.
Then there exists $\beta < \alpha$ such that
$N (\beta) \leq f [n] [N [\alpha]] (m)$ and 
$f [n]^\alpha (m) =
 f [n]^\beta (f [n]^\beta (m))$.
Let us observe that
\begin{eqnarray}
N(\beta) = f(n + N(\alpha) + m)
&\leq&
f(N(\alpha) + 2m)
\quad \text{ since } n \leq m, 
\nonumber \\
&\leq&
f(N(\alpha) + f(m)) \quad \text{ from } (f.\ref{f:1}). 
\label{e:lem:f[n](m)}
\end{eqnarray}
We can see that
$f [n]^\alpha (f(m)) \leq f^\alpha [f^\alpha (f(m))] (f(m))$
holds as follows.
\begin{eqnarray*}
f [n]^\alpha (m) &\leq&
f^\beta (f^\beta (f(m)) + f^\beta (f^\beta (f(m)) + f(m)))
\quad \text{ by IH,} \\
&\leq&
f^\beta (f^\beta (2 \cdot f^\beta (f(m)) + f(m)))
\quad \text{ by } (f^\beta.\ref{f:1}), \\
&\leq&
f^\beta (f^\beta (f^\beta (f^\beta (f(m))) + f(m)))
\quad \text{ by } (f^\beta.\ref{f:2}), \\
&\leq&
f^\beta (f^\beta (f^\alpha (f(m))
                 ) + f(m)
        )
\quad \text{ by (\ref{e:lem:f[n](m)}),} \\
&\leq&
f^\alpha (f^\alpha (f(m)) + f(m)) =
f^\alpha [f^\alpha (f(m))] (f(m)).
\end{eqnarray*}
The last inequality holds since
$N(\beta) \leq f(N(\alpha) + f^\alpha (f(m)) + f(m))$.
\qed
\end{proof}

\begin{corollary}
\label{c:f[n][m]}
If $n \leq m$, then
$(f[n])^\alpha (m) \leq f^{\alpha +2} (m)$.
\end{corollary}

\begin{proof}
By Lemma \ref{lem:f[n](m)}.\ref{lem:f[n](m):2},
$f[n]^\alpha (m) \leq 
 f^\alpha (f^\alpha (f(m)) + f(m)) \leq
 f^\alpha (f^\alpha (2 \cdot f(m))) \leq
 f^{\alpha +1} (f^0(f^0 (m))) \leq
 f^{\alpha +1} (f^{\alpha +1} (m)) \leq
 f^{\alpha +2} (m)$.
\qed
\end{proof}

We define a relation
$
 f, F \vdash^\alpha_\rho \Gamma$
for a quintuple
$(f, F, \alpha, \rho, \Gamma )$
where
$\alpha < \varepsilon_{\Omega +1}$, $\rho < \Omega \cdot \omega$ and 
$\Gamma$ is a sequent of $\Lstar$-sentences.
In this paper a ``sequent'' means a finite set of formulas.
We write $\Gamma, A$ or $A, \Gamma$ to denote
$\Gamma \cup \{ A \}$.
Let us recall that for a finite set $K \subseteq \Ord$,
$F[K](\xi)$ denotes $F(\max (K \cup \{ \xi \}))$.
We will write
$F[\mu](\xi)$ to denote $F[\{ \mu \}] (\xi)$.
We write $\mathsf{TRUE}_0$ to denote the set 
$\{ A \mid A \text{ is an } 
    \mathcal{L}_{\mathrm{PA}} \text{-literal true in the standard model }
    \mathbb N \text{ of } \mathrm{PA}
 \}$.

\begin{definition}
\label{d:OCD}
$f, F \vdash^\alpha_\rho \Gamma$ if 
\begin{equation}
\tag{$\hyp{}{}(f;F;\alpha)$}
\max \{ N (F (0)), N (\alpha)\} 
       \} 
\leq f(0), \quad
K_{\Omega} \alpha < F (0), 
\label{hyp:1}
\end{equation}
and one of the following holds.
\begin{description}
\item[$(\mathsf{Ax}1)$]  
$\exists A(x)$: an $\LID$-literal,
$\exists s, t \in \mathcal{T} (\LID)$ s.t. $\mathsf{FV} (A) = \{ x \}$,
$\val (s) = \val (t)$
and
$\{ \neg A(s), A(t) \} \subseteq \Gamma$,.
\item[$(\mathsf{Ax}2)$] $\Gamma \cap \mathsf{TRUE}_0 \neq \emptyset$.
\item[$(\bigvee)$] 
$\exists A \simeq \OR{\iota \in J} A_\mu \in \Gamma$, 
$\exists \alpha_0 < \alpha$, $\exists \iota_0 \in J$ s.t. 
$N (\iota_0) \leq f(0)$
$\matho (\iota_0) < \min \{ \alpha, F(0) \}$,
and 
$f, F \vdash^{\alpha_0}_\rho \Gamma, A_{\iota_0}$.
\item[$(\bigwedge)$] 
$\exists A \simeq \AND{\iota \in J} A_\iota \in \Gamma$
s.t.
$N(\max \kPO (A)) \leq f(0)$,
$\kPO (A) < F(0)$ and
$(\forall \iota \in J)$ $(\exists \alpha_\iota < \alpha)$ 
$[f [N (\iota)], F [\matho (\iota)] \vdash^{\alpha_\iota}_\rho \Gamma, A_\iota]$.
\item[$(\Clrule)$] 
$\exists t \in \mathcal{T}(\LID)$,
$\exists \alpha_0 < \alpha$ s.t.
$\itP{\ofA}{\Omega} t \in \Gamma$,
$\Omega < \alpha$ 
and 
$f, F \vdash^{\alpha_0}_\rho \Gamma, \ofA (\itP{\ofA}{\Omega}, t)$.
\item[$(\mathsf{Cut})$] 
$\exists C$: an $\Lstar$-sentence of $\OR{}$-type,
$\exists \alpha_0 < \alpha$ 
s.t.
$\max \{ \lh (C), N(\max \kPO (C))$,
$         N(\max (\kSO (C))
      \} \leq f(0)$,
$\kO (C) < F(0)$,
$\rk (C) <\rho$,
$f, F \vdash^{\alpha_0}_\rho \Gamma, C$, and
$f, F \vdash^{\alpha_0}_\rho \Gamma, \neg C$.
\end{description}

We will call the pair $(f, F)$ {\em operators} controlling the derivation
that forms $f, F \vdash^\alpha_\rho \Gamma$.
\end{definition}

In the sequel we always assume that the operator $F$ enjoys the
following condition ($\hyp{}{}(F)$):  
\begin{equation}
\tag{$\hyp{}{}(F)$}
\eta < F (\xi) \Rightarrow 
F (\eta) \leq F (\xi) 
\quad \text{ for any ordinals } \xi, \eta < \Omega.
\label{hyp:2}
\end{equation}

We note that the hypothesis ($\hyp{}{}(F)$) reflects the fact stated in
Proposition \ref{p:F^a(eta)}. 
It is not difficult to see that if the condition
$(\hyp{}{}(F))$ holds, then the condition $(\hyp{}{}(F[K]))$ also holds for
any finite set $K < \Omega$.

\begin{lemma}[Inversion]
\label{lem:inversion}
Assume that $A \simeq \AND{\iota \in J} A_\iota$.
If $f, F \vdash^\alpha_\rho \Gamma, A$, then
$f [N (\iota)], F [ \matho (\iota)] \vdash^\alpha_\rho \Gamma, A_\iota$
for all $\iota \in J$.
\end{lemma}

\begin{proof}
By induction on $\alpha$.
Let $\iota \in J$.
Then we can check that the condition
$\hyp{}{} (f[N(\iota)]; F[\matho (\iota)]; \alpha)$ holds.
In particular, by the hypothesis $\hyp{}{} (f; F; \alpha)$ 
we have 
$N (F [\matho (\iota)]) = N (\iota) +  N (F(0)) \leq 
 N(\iota) + f(0) \leq f [N(\iota)] (0)$.
Now the assertion is a straightforward consequence of IH.
\qed
\end{proof}

We write $f \circ g$ to denote the result 
$m \mapsto f(g(m))$ of composing $f$ and $g$.

\begin{lemma}[Cut-reduction]
\label{lem:cut-red}
Assume that $C \simeq \OR{\iota \in J} C_\mu$,
$\rk (C) = \rho \neq \Omega$,
$\max \{ \lh (C), N(\max \kPO (C)), N(\max \kSO (C)) \} \leq f (g (0))$, 
and that
$\kO (C) < F(0)$.
If 
$f, F \vdash^\alpha_\rho \Gamma, \neg C$ and
$g, F \vdash^\beta_\rho \Gamma, C$, then
$f \circ g, F \vdash^{\alpha + \beta}_\rho \Gamma$.
\end{lemma}

\begin{proof}
By induction on $\beta$.

\textsc{Case.} $C$ is not the principal formula of the last rule
 ($\infJ$) that forms 
$g, F \vdash^\beta_\rho \Gamma, C$:
We only consider the case that ($\infJ$) is ($\AND{}$).
The other cases can be treated similarly.
Let us suppose that the sequent $\Gamma$ contains a formula
$\AND{\iota \in J} A_\iota$ and and the inference rule ($\infJ$) has
 the premises
$g[N(\iota)], F[\matho (\iota)] \vdash^{\beta_\iota}_\rho 
 \Gamma, A_\iota,  C$ $(\iota \in J)$
for some $\beta_\iota < \beta$.
Then, since 
$f \circ (g[N(\iota)] (0)) = (f \circ g) [N(\iota)] (0)$
and
$F(0) \leq F[\matho (\iota)] (0)$ for all $\iota \in J$, IH
 yields the sequent 
\begin{equation*} 
(f \circ g) [N(\iota)], F[\matho (\iota)] 
\vdash^{\alpha + \beta_\iota}_\rho \Gamma, A_\iota
\end{equation*}
for all $\iota \in J$.
Hence another application of ($\AND{}$) yields the sequent
$f \circ g, F \vdash^{\alpha + \beta}_{\rho} \Gamma$.

\textsc{Case.} $C$ is the principal formula of the last rule ($\infJ$):
In this case ($\infJ$) should be ($\bigvee$) since 
$\rk (C) \neq \Omega$.
Let the premise be of the form
$g, F \vdash^{\beta_0}_\rho \Gamma, C_{\iota_0}, C$ for some
$\beta_0 < \beta$ and $\iota_0 \in J$ such that
$N(\iota_0) \leq g(0)$ and
$\matho (\iota_0) < \min \{ \beta, F(0) \}$.
IH yields the sequent
\begin{equation}
f \circ g, F \vdash^{\alpha + \beta_0}_\rho \Gamma, C_{\iota_0}.
\label{e:l:cut-red:1}
\end{equation}
On the other hand, Inversion lemma yields
 the sequent
$f [N (\iota_0)], F [ \matho (\iota_0)] \vdash^\alpha_\rho 
 \Gamma, \neg C_{\iota_0}$.
Let us observe the following.
First, 
$f [N (\iota_0)] (0) = f (N (\iota_0)) \leq f (g(0)) = 
 (f \circ g) (0)$
since $N (\iota_0) \leq g(0)$.
Secondly, 
$F [\matho (\iota_0)] (0) \leq F (0)$ 
 by the hypothesis $\hyp{}{} (F)$ since $\matho (\iota_0) < F (0)$.
Hence
\begin{equation}
f \circ g, F \vdash^{\alpha + \beta_0}_\rho, 
 \Gamma, \neg C_{\iota_0}.
\label{e:l:cut-red:2}
\end{equation}
We also observe that
$N (\alpha + \beta) \leq N (\alpha) + N (\beta) \leq
 f(0) + g(0) \leq (f \circ g) (0)$. 
Further
$K_{\Omega} (\alpha + \beta) < F (0)$
since 
$ K_{\Omega} \alpha \cup K_{\Omega} \beta < F(0)$.
Now by an application of $(\mathsf{Cut})$ to the two sequents (\ref{e:l:cut-red:1}) and
 (\ref{e:l:cut-red:2}) we obtain
$f \circ g, F \vdash^{\alpha + \beta}_\rho \Gamma$.

The other cases are similar.
\qed
\end{proof}

For a sequent $\Gamma$
we write $\kPO (\Gamma)$ to denote the set
$\bigcup_{B \in \Gamma} \kPO (B)$.

\begin{lemma}
\label{l:CE1}
Let $k < \omega$.
If $f, F \vdash^\alpha_{\Omega +k+2} \Gamma$, then
$f^{F^\alpha (0) +1}, F
 \vdash^{\Omega^\alpha}_{\Omega +k+1} \Gamma$.
\end{lemma}

\begin{proof}
By induction on $\alpha$. The argument splits into several cases
 depending on the last rule that forms 
$f, F \vdash^\alpha_{\Omega +k+2} \Gamma$.
We only consider the following two critical cases.
Let $K$ denote the set $\kPO (\Gamma)$.

\textsc{Case.} The last rule is $(\mathsf{Cut})$: In this case there are two
 premises 
$f, F \vdash^{\alpha_0}_{\Omega +k+2} \Gamma, C$ and
$f, F \vdash^{\alpha_0}_{\Omega +k+2} \Gamma, \neg C$ with
a cut formula $C$ for some $\alpha_0 < \alpha$ such that
$\rk (C) < \Omega + k+2$,
$\max \{ \lh (C), N(\max \kPO (C)), N(\max \kSO (C)) \} \leq f(0)$ 
and
$\kO (C) < F(0)$.
Let $K_0$ denote the set $\kPO (\Gamma, \neg C)$.
Then IH yields the two sequents 
\[
 f^{F^{\alpha_0} [K_0](0) +1}, F 
 \vdash^{\Omega^{\alpha_0}}_{\Omega +k+1}
 \Gamma, C,
 \quad
 f^{F^{\alpha_0} [K_0](0) +1}, F \vdash^{\Omega^{\alpha_0}}_{\Omega +k+1}
 \Gamma, \neg C.
\]Hence Cut-reduction lemma yields the sequent
\begin{equation*}
f^{F^{\alpha_0} [K_0](0) +1} \circ f^{F^{\alpha_0} [K_0](0) +1}, F 
 \vdash^{\Omega^{\alpha_0} + \Omega^{\alpha_0}}_{\Omega +k+1}
 \Gamma.
\end{equation*}
Clearly $\Omega^{\alpha_0} + \Omega^{\alpha_0} < \Omega^\alpha$.
Further
$N (\Omega^\alpha) = N (\alpha) +1 \leq f^{F^{\alpha}[K](0) +1} (0)$ 
since
$N (\alpha) \leq f(0) = f^0 (0) < f^{F^{\alpha} (0) +1} (0)$.
It remains to show that 
\[
 (f^{F^{\alpha_0} [K_0](0) +1} \circ f^{F^{\alpha_0} [K_0](0) +1}) (0) \leq 
 f^{F^{\alpha} [K](0) +1} (0).
\]
Let us see that
$K_0 \subseteq K \cup \kO (C)
 < F^\alpha [K] (0)$
since $\kO (C) < F(0)$.
This implies $F^{\alpha_0}[K_0](0) < F^\alpha [K](0)$,
and hence
$F^{\alpha_0}[K_0](0) + 1 < F^\alpha [K](0)$.
We can also see that
\[
 N(\max K_0) \leq 
 \max \{ N(\max K), N(\max \kSO (C)) \} \leq 
 \max \{ N(\max K), f(0) \}.
\]
From this and the inequality $N(\alpha_0) \leq f(0)$ one can see that
\begin{eqnarray*}
N(F^{\alpha_0}[K_0](0) +1) 
&\leq&
N(F[K_0](0)) + N(\alpha_0) + 1 \\
&\leq&
N(F[K](0)) + f(0) + f(0) +1 \\
&\leq&
f(N(F^{\alpha}[K](0)) + f^{F^\alpha [K](0)}(0)).
\end{eqnarray*}
This allows us to conclude as follows.
\begin{eqnarray*}
&&
(f^{F^{\alpha_0} [K_0](0) +1} \circ f^{F^{\alpha_0} [K_0](0) +1}) (0) \\
&\leq& 
(f^{F^{\alpha_0} [K_0](0) +1} \circ f^{F^{\alpha_0} [K_0](0) +1}) 
(f^{F^\alpha [K](0)}(0)) \\
&\leq& 
f^{F^{\alpha} [K](0)} (f^{F^\alpha [K](0)}(0)) \\
&\leq&
f^{F^\alpha [K](0) +1}(0).
\end{eqnarray*}

\textsc{Case.} The last rule is ($\bigwedge$):
In this case there exists a formula
$A \simeq \AND{\iota \in J} A_\iota \in \Gamma$ such that
$N(\max \kPO (A)) \leq f(0)$,
$\kPO (A) < F(0)$ and
$\forall \iota \in J$, $\exists \alpha_\iota < \alpha$ s.t. 
$f [N (\iota)], F [ \matho (\iota)] 
 \vdash^{\alpha_\iota}_{\Omega +k+2} \Gamma, A_\iota$.
By IH 
$(f [N (\iota)])^{F [\matho (\iota)]^{\alpha_\iota}(0) +1}, 
 F [\matho (\iota)] 
 \vdash^{\Omega^{\alpha_\iota}}_{\Omega +k+1} \Gamma, A_\iota$
for all $\iota \in J$.

\begin{claim}
$(f [N (\iota)])^{F [\matho (\iota)]^{\alpha_\iota}(0) +1} (0) \leq
 f^{F^\alpha (0) +1} [N (\iota)] (0)$ for all $\iota \in J$.
\end{claim}
\label{c:CE1}

Assuming the claim,
$f^{F^\alpha (0) +1} [N (\iota)], F [ \matho (\iota)] 
 \vdash^{\Omega^{\alpha_\iota}}_{\Omega +k+1} \Gamma, A_\iota$
for all $\iota \in J$ and hence an application of ($\bigwedge$) yields
$f^{F^\alpha (0) +1}, F \vdash^{\Omega^\alpha}_{\Omega +k+1} \Gamma$.
To show the claim fix $\iota \in J$ arbitrarily and let 
$n := N (\iota)$. 
Then Corollary \ref{c:f[n][m]} yields
\begin{equation}
f[n]^{F [\matho (\iota)]^{\alpha_\iota} (0) +1} (0) \leq 
f^{F [\matho (\iota)]^{\alpha_\iota}(0) +3} (n).
\label{e:CE1:1}
\end{equation}
By Lemma \ref{lem:kPi}.\ref{lem:kPi:2}, 
$\matho (\iota) \leq \kPO (A)$
since $\matho (\iota) < \Omega$.
Hence
$\matho (\iota) < F(0)$
since
$\kPO (A) < F(0)$.
This together with the hypothesis ($\hyp{}{}(F)$) yields
$K_{\Omega} \alpha_\iota < F[\matho (\iota)] \leq F (0) \leq
 F^\alpha (0)$.
Further
$F [\matho (\iota)]^{\alpha_\iota} (0) \leq
 F^{\alpha_\iota} (\matho (\iota))$
by Lemma \ref{lem:F[xi]}.
Hence
$F [\matho (\iota)]^{\alpha_\iota}(0) = 
 F^{\alpha_\iota} (\matho (\iota)) < F^\alpha (0)$
since $\matho (\iota) < F(0) \leq F^\alpha (0)$.
And hence
\begin{equation}
F [\matho (\iota)]^{\alpha_\iota}(0) + 3 < F^\alpha (0).
\label{e:CE1:2}
\end{equation}
As in Example \ref{ex:1} we can see that
$2n +3 \leq f^\omega (n) \leq f^{F^\alpha (0)} (n)$.
Hence
\begin{eqnarray}
&&
N(F [\matho (\iota)]^{\alpha_\iota} (0) +3) \nonumber \\
&=&
N(F(0)) + N(\matho (\iota)) + N(\alpha_\iota) +3 \nonumber \\
&\leq&
N(F^\alpha (0)) +n + f(n) +3 
\quad \text{ since } 
N (\alpha_\iota) \leq f[N(\iota)] (0) = f(n), \nonumber \\
&\leq&
f(N(F^\alpha (0)) + 2n + 3) 
\quad \text{ from the condition } (f.\ref{f:1}), \nonumber \\
&\leq&
f(N(F^\alpha (0)) + f^{F^\alpha (0)} (n)).
\label{e:CE1:3}
\end{eqnarray}
The two conditions (\ref{e:CE1:2}) and  (\ref{e:CE1:3}) allows us to
 deduce that
\begin{eqnarray}
f^{F [\matho (\iota)]^{\alpha_\iota} (0) +3} (n) 
&\leq&
f^{F [\matho (\iota)]^{\alpha_\iota} (0) +3} 
(f^{F^\alpha (0)} (n)) \nonumber \\
&\leq&
f^{F^\alpha (0)} (f^{F^\alpha (0)} (n)) \nonumber \\
&\leq&
f^{F^{\alpha +1} (0)} (n) =
f^{F^{\alpha +1} (0)} [n] (0).
\label{e:CE1:4}
\end{eqnarray}
Combining the two inequality (\ref{e:CE1:1}) and  (\ref{e:CE1:4})
 enables us to conclude the claim, and hence completes this case.
\qed
\end{proof}

\begin{lemma}[Predicative Cut-elimination]
\label{lem:PCE}
Assume 
$\{ \alpha, \beta, \gamma \} < \Omega$,
$N (\alpha) \leq f^\gamma (0)$ and
$K_{\Omega} \alpha < F (0)$.
If 
$f^\gamma, F \vdash^\beta_{\rho + \omega^\alpha} \Gamma$, then
$f^{F^{\Omega \cdot \alpha + \gamma + \beta} (0) +1}, F
 \vdash^{\varphi \alpha \beta}_{\rho} \Gamma$.
\end{lemma}

\begin{proof}
By main induction on $\alpha$ and side induction on $\beta$.
Let us start with observing the following.
First 
$N (\varphi \alpha \beta) = N (\alpha) + N(\beta) +1 \leq
 f^\gamma (0) + f^\gamma (0) +1 \leq f^\gamma (f^\gamma (0)) +1 \leq
 f^{F^{\Omega \cdot \alpha + \gamma + \beta} (0) +1} (0)$. 
Secondly 
$K_{\Omega} \varphi \alpha \beta = \{ \varphi \alpha \beta \} < F(0)$ 
since 
$K_{\Omega} \alpha \cup K_{\Omega} \beta < F(0)$.

\textsc{Case.}
The last rule is ($\bigwedge$):
In this case there exists a formula
$A \simeq \AND{\iota \in J} A_\iota \in \Gamma$ and for all $\iota \in J$
 there exists $\beta_\iota < \beta$ such that
$f^\gamma [N (\iota)], F[ \matho (\iota)] 
 \vdash^{\beta_\iota}_{\rho + \omega^\alpha} \Gamma, A_\iota$.
We observe that
$N (\alpha) \leq f^\gamma (0) \leq f [N (\iota)]^\gamma (0)$ 
and 
$K_{\Omega} \alpha < F (0) \leq F [\matho (\iota)] (0)$.
Hence Side Induction Hypothesis yields that for all $\iota \in J$
\begin{equation}
f[N (\iota)]^{F^{\Omega \cdot \alpha + \gamma + \beta_\iota} (0) +1}, 
F[ \matho (\iota)] 
 \vdash^{\varphi \alpha \beta_\iota}_{\rho} \Gamma, A_\iota.
\label{e:l:PCE:1}
\end{equation}
Let $m := N (\iota)$.
Then 
$f [m]^{F^{\Omega \cdot \alpha + \gamma + \beta_\iota} (0) +1} (0) \leq
 f^{F^{\Omega \cdot \alpha + \gamma + \beta_\iota} (0) +3}  [m] (0)$
from Corollary \ref{c:f[n][m]}.
Also it holds that 
$F^{\Omega \cdot \alpha + \gamma + \beta_\iota} (0) < 
 F^{\Omega \cdot \alpha + \gamma + \beta} (0)$ 
for all $\iota \in J$ since
$K_{\Omega} \beta_{\iota} < F[\matho (\iota)] (0) \leq F(0)$.
Further
\begin{eqnarray}
N (F^{\Omega \cdot \alpha + \gamma + \beta_\iota} (0) +3) &=&
 N (F(0)) + N (\alpha) + N (\gamma) + N (\beta_\iota) +4 
\nonumber\\
&\leq&
2 \cdot f^\gamma (0) + f^{F^\gamma (0)} (0) 
 + f^\gamma [m] (0) + 4 
\quad \text{ by Lemma \ref{lem:N(a)},}
\nonumber \\
&\leq& 
 f^{F^\gamma (0)} (f^\gamma (f^\gamma ( f^\gamma (m)))) + 4 
\nonumber \\
&\leq&
f^{F^\gamma (0)} (f^{\gamma +2} (m)) +4 \nonumber \\
&\leq&
f^{F^\gamma (0)} (f^{\gamma +2} (m)) + f^{\gamma +2}(0) \nonumber \\
&\leq&
f^{F^\gamma (0)} (f^{\gamma +2} (m) + f^{\gamma +2}(0)) \nonumber \\
&\leq&
f^{F^\gamma (0)} (f^{\gamma +3} (m)) \nonumber \\
&\leq&
f^{F^\gamma (0) +2} (m) \nonumber \\ 
&\leq&
 f^{F^{\Omega \cdot \alpha + \gamma + \beta (0)}} (m). 
\label{e:l:PCE:2}
\end{eqnarray}
The last inequality holds since 
$N(F^\gamma (0) +2) = N(F(0)) + N(\gamma) +2$ 
is bounded by
$f[N(F^{\Omega \cdot \alpha + \gamma + \beta (0)})] (m)$.
Hence
\begin{eqnarray}
f^{F^{\Omega \cdot \alpha + \gamma + \beta_\iota} (0) +3} (m) &\leq&
f^{F^{\Omega \cdot \alpha + \gamma + \beta_\iota} (0) +3}
  (f^{F^{\Omega \cdot \alpha + \gamma + \beta} (0)} (m)) 
\nonumber \\
&\leq&
f^{F^{\Omega \cdot \alpha + \gamma + \beta} (0)}
  (f^{F^{\Omega \cdot \alpha + \gamma + \beta} (0)} (m)) 
\quad \text{ by (\ref{e:l:PCE:2}),} 
\nonumber \\
&\leq&
 f^{F^{\Omega \cdot \alpha + \gamma + \beta} (0) +1} (m). 
\nonumber
\end{eqnarray}
This together with (\ref{e:l:PCE:1}) allows us to derive the sequent
\begin{equation*}
f^{F^{\Omega \cdot \alpha + \gamma + \beta} (0) +1} [N (\iota)], 
 F[ \matho (\iota)] 
 \vdash^{\varphi \alpha \beta_\iota}_{\rho} \Gamma, A_\iota.
\end{equation*}
An application of ($\bigwedge$) yields
$f^{F^{\Omega \cdot \alpha + \gamma + \beta} (0) +1}, F
 \vdash^{\varphi \alpha \beta}_{\rho} \Gamma, A$.

\textsc{Case.}
The last rule is $(\mathsf{Cut})$:
In this case there exist a formula $C$ and an ordinal
$\beta_0 < \beta$  such that
$\rk (C) < \rho + \omega^\alpha$, 
$\max \{ \lh (C), N (\max \kPO (C)), N (\max \kSO (C)) \}$ $\leq f^\gamma (0)$,
$\kO (C) < F(0)$,
\begin{equation*}
f^\gamma, F \vdash^{\beta_0}_{\rho + \omega^\alpha}
 \Gamma, C
\quad \text{ and } \quad
f^\gamma, F \vdash^{\beta_0}_{\rho + \omega^\alpha}
 \Gamma, \neg C.
\end{equation*}
SIH yields
$f^{F^{\Omega \cdot \alpha + \gamma + \beta_0} (0) +1}, F
 \vdash^{\varphi \alpha \beta_0}_{\rho} \Gamma, C$
and
$f^{F^{\Omega \cdot \alpha + \gamma + \beta_0} (0) +1}, F
 \vdash^{\varphi \alpha \beta_0}_{\rho} \Gamma, \neg C$.
If $\rk (C) < \rho$, then we can apply ($\mathsf{Cut}$), having the conclusion.
Suppose that 
$\rho \leq \rk (C) < \rho + \omega^{\alpha}$.
Then there exist $l < \omega$ and $\alpha_1, \dots, \alpha_l$ such
 that $\alpha_l \leq \cdots \leq \alpha_1 < \alpha $ and 
$\rk (C) = \rho + \omega^{\alpha_1} + \cdots + \omega^{\alpha_l}$.
Let 
$\gamma' := F^{\Omega \cdot \alpha + \gamma + \beta_0} (0) +2$.
Then it is easy to observe that
$f^{F^{\Omega \cdot \alpha + \gamma + \beta_0} (0) +1}
 (f^{F^{\Omega \cdot \alpha + \gamma + \beta_0} (0) +1} (m)) \leq
 f^{\gamma'} (m)$
for all $m$.
This together with Cut-reduction lemma (Lemma \ref{lem:cut-red}) yields 
\begin{equation}
f^{\gamma'}, F
 \vdash^{\varphi \alpha \beta_0 + \varphi \alpha \beta_0
        }_{\rho + \omega^{\alpha_1} \cdot l}
 \Gamma.
\label{e:l:PCE:3}
\end{equation}
Let us define ordinals $\xi_n$ and $\gamma_n$ by
\[
 \left\{
 \begin{array}{rcl}
 \xi_0 &=& \varphi \alpha \beta_0 + \varphi \alpha \beta_0, \\
 \xi_{n+1} &=& \varphi \alpha_1 \xi_n,   
 \end{array}
 \right.
\quad
 \left\{
 \begin{array}{rcl}
 \gamma_0 &=& \gamma' = 
 F^{\Omega \cdot \alpha + \gamma + \beta_0} (0) +2, \\
 \gamma_{n+1} &=& F^{\Omega \cdot \alpha_1 + \gamma_n + \xi_n} (0) +1. 
 \end{array}
 \right.
\]

\begin{claim}
$f^{\gamma_n}, F \vdash^{\xi_n}_{\rho + \omega^{\alpha_1} \cdot (l-n)}
\Gamma$.
$(0 \leq n \leq l)$
\end{claim}

We show the claim by subsidiary induction on $n \leq l$.
The base case follows immediately from (\ref{e:l:PCE:3}).
For the inductions step suppose $n < l$.
Then by IH we have
$f^{\gamma_n}, F 
 \vdash^{\xi_n}_{\rho + \omega^{\alpha_1} (l-(n+1)) + \omega^{\alpha_1}}
\Gamma$. 
It is easy to see that
$\{ \alpha_1, \xi_n, \gamma_n \} < \Omega$ and that
$\gamma < \gamma_m$ and 
$N (\gamma) \leq N (\gamma_m)$ for all $m \leq l$.
Hence
\[
 \left\{
 \begin{array}{l}
 N (\alpha_1) \leq N (\rk (C)) \leq f^\gamma (0) \leq
 f^{\gamma_n} (0), \\
 K_{\Omega} \alpha_1 \subseteq K_{\Omega} \alpha < F(0). 
 \end{array} 
 \right.
\]
Thus MIH of the lemma yields
$f^{\gamma_{n+1}}, F 
 \vdash^{\xi_{n+1}}_{\rho + \omega^{\alpha_1} (l-(n+1))}
\Gamma$. 
\qed

By the claim with $n = l$ we have
$f^{\gamma_{l}}, F \vdash^{\xi_{l}}_{\rho} \Gamma$. 
One can show
$\xi_n < \varphi \alpha \beta$ by a straightforward induction on $n$.
Hence
$\xi_{l} < \varphi \alpha \beta$.
It remains to show that 
$f^{\gamma_{l}} (0) \leq 
 f^{F^{\Omega \cdot \alpha + \gamma + \beta} (0) +1} (0)$.
It is not difficult to check
$\gamma_{l} < F^{\Omega \cdot \alpha + \gamma + \beta} (0) +1$.
By simultaneous induction on $n$ we show the following
(\ref{eq:N(xi_n)}) and (\ref{eq:N(gamma_n)}):
\begin{eqnarray}
N (\xi_n) &\leq&
n N (\alpha_1) + 2 N (\alpha) + 2 N (\beta_0) +2 + n, 
\label{eq:N(xi_n)}
\\
N (\gamma_n) &\leq&
(n+1) N (F(0)) + \frac{1}{2}n(n+1) N (\alpha_1) 
\nonumber \\
&&
+ (2n +1) N (\alpha) + N (\gamma) +
(2n +1) N (\beta_0) + 4 (n+1).
\label{eq:N(gamma_n)}
\end{eqnarray}
For the base case
\begin{eqnarray*}
N (\xi_0) &\leq& 2 (N (\alpha) + N (\beta_0) +1) \leq
 2 N (\alpha) + 2 N (\beta_0) +2, \\
N (\gamma_0) &\leq& 
 N (F(0)) + N (\alpha) + N (\gamma) + N (\beta_0) +4.
\end{eqnarray*}
Let us consider the induction step.
Assuming (\ref{eq:N(xi_n)}),
\begin{eqnarray*}
N (\xi_{n+1}) &=& N (\alpha_1) + N (\xi_n) +1 \\
&\leq&
 (n+1) N (\alpha_1) + 2 N (\alpha) + 2 N (\beta_0) +2 + n+1.
\end{eqnarray*}
Assuming both (\ref{eq:N(xi_n)}) and (\ref{eq:N(gamma_n)}),
\begin{eqnarray*}
N (\gamma_{n+1}) &=&
 N (F(0)) + N (\alpha_1) + N (\gamma_n) + N (\xi_n) + 4 \\
&\leq&
 (n+2) N (F(0)) + (\frac{1}{2} n (n+1) + n+1) N (\alpha_1) \\
&& +
 (2n +3) N (\alpha) + N (\gamma) + (2n +3) N (\beta_0) +
 4 (n+1) +4 \\
&\leq&
 (n+2) N (F(0)) + \frac{1}{2} (n+1)(n+2) N (\alpha_1) \\
&& +
 (2n +3) N (\alpha) + N (\gamma) + (2n +3) N (\beta_0) +
 4 (n+2).
\end{eqnarray*}
Let us observe that
\begin{eqnarray}
N (\rk (C)) &\leq&
N(\max \kPS (C)) + \lh (C)
\quad \text{ by Lemma \ref{lem:kPi}.\ref{lem:kPi:3},} 
\nonumber \\
&\leq&
f^\gamma (0) + f^\gamma (0) \quad 
\text{ since } \lh (C)  \leq f^\gamma (0),
\nonumber \\
&\leq&
f^\gamma (f^\gamma (0)).
\label{e:l:PCE:6}
\end{eqnarray}
Hence 
$l \leq \rk (C) \leq f^\gamma (f^\gamma (0)) \leq f^{F^\gamma (0)} (0)$
since $\gamma < F^\gamma (0)$ and 
$N (\gamma) \leq N (F^\gamma (0))$.
Further
$\max \{ F(0), N(\alpha), N(\beta_0) \} \leq f^\gamma (0) \leq
 f^{F^\gamma (0)}$ 
by assumption and 
$N(\gamma) \leq f^{F^\gamma (0)} (0)$
by Lemma \ref{lem:N(a)}.
From these and (\ref{eq:N(gamma_n)}),
\begin{equation}
N(\gamma_l) \leq
(f^{F^\gamma (0)} (0))^3 + 6 (f^{F^\gamma (0)} (0))^2 +
8 \cdot f^{F^\gamma (0)} (0) + 4.
\label{e:l:PCE:7}
\end{equation}
On the other hand, from Example \ref{ex:1}, one can see
 that 
$m^3 + 6 m^2 + 8 m + 4 \leq f^{F^\gamma (0)} (m)$
holds.
Hence by (\ref{e:l:PCE:7}),
\begin{equation}
N(\gamma_l) \leq f^{F^\gamma (0)} (f^{F^\gamma (0)} (0)) \leq 
 f^{F^{\Omega \cdot \alpha + \gamma + \beta} (0)} (0) \leq
f 
  (f^{F^{\Omega \cdot \alpha + \gamma + \beta} (0)} (0)).
\label{e:l:PCE:8}
\end{equation}
Hence
\begin{eqnarray*}
f^{\gamma_{l}} (0) &\leq& 
f^{\gamma_{l}}
 (f^{F^{\Omega \cdot \alpha + \gamma + \beta} (0)} (0)) \\
&\leq& 
 f^{F^{\Omega \cdot \alpha + \gamma + \beta} (0)}
   (f^{F^{\Omega \cdot \alpha + \gamma + \beta} (0)}) 
\quad \text{ by (\ref{e:l:PCE:8}),} \\
&\leq&
 f^{F^{\Omega \cdot \alpha + \gamma + \beta} (0) +1} (0).
\end{eqnarray*}
This allows us to conclude
$f^{F^{\Omega \cdot \alpha + \gamma + \beta} (0) +1}, F
 \vdash^{\varphi \alpha \beta}_{\rho} \Gamma$.
\qed
\end{proof}

\begin{definition}
For each $\Lstar$-formula $B$ let $B^\alpha$ be the result of
 replacing in $B$ every occurrence of $\itP{\ofA}{\Omega}$ by
$\itP{\ofA}{\alpha}$. 
\end{definition}

\begin{lemma}[Boundedness]
\label{lem:reflection}
Assume that
$f, F \vdash^{\alpha}_{\rho} \Gamma, A$.
Then for all $\xi$ if $\alpha \leq \xi \leq F (0)$,
$N (\xi) \leq f(0)$ and 
$K_{\Omega} \xi < F(0)$, then
$f, F \vdash^{\alpha}_{\rho} \Gamma, A^\xi$.
\end{lemma}

\begin{proof}
The claim is trivial if $F (0) < \alpha$.
Assume that $\alpha \leq F (0)$ and 
$f, F \vdash^{\alpha}_{\rho} \Gamma, A$.
By induction on $\alpha$ we show that for all $\xi$ if 
$\alpha \leq \xi \leq F (0)$, then
$f, F \vdash^{\alpha}_{\rho} \Gamma, A^\xi$.

\textsc{Case.}
The last rule is ($\bigvee$):
If $A$ is not the principal formula of last rule ($\bigvee$), then the claim follows
 immediately from IH.
Suppose that $A \simeq \OR{\iota \in J} A_\iota$ is the principal formula of
 ($\bigvee$). 
Then there exist $\alpha_0 < \alpha$ and $\iota_0 \in J$ such that
$\matho (\iota_0) < \alpha$ and 
$f, F \vdash^{\alpha_0}_\rho \Gamma, A, A_{\iota_0}$.
Let $\alpha \leq \xi \leq F (0)$.
Then IH yields 
$f, F \vdash^{\alpha_0}_\rho \Gamma, A^\xi, A_{\iota_0}$.
If $A \not\equiv \itP{\ofA}{\Omega} t$, then another application of IH and an
 application of ($\bigvee$) yield
$f, F \vdash^{\alpha}_\rho \Gamma, A^\xi$.
Consider the case that 
$A \equiv \itP{\ofA}{\Omega} t \simeq 
 \OR{\mu < \Omega} \ofA (\itP{\ofA}{\mu}, t)$.
In this subcase $A_{\mu_0} \simeq \ofA (\itP{\ofA}{\mu_0}, t)$.
Since $\mu_0 = \matho (\mu_0) < \alpha \leq \xi$, we can apply ($\bigvee$) and then
 obtain $f, F \vdash^\alpha_\rho \Gamma, \itP{\ofA}{\xi}$.

\textsc{Case.}
The last rule is ($\bigwedge$):
In this case for all $\iota \in J$ there exists 
$\alpha_\iota < \alpha$ such that 
$f[N (\iota)], F [\matho (\iota)] 
 \vdash^{\alpha_\iota}_{\rho} \Gamma'$ for a certain $\Gamma'$.
Let us observe that 
$F (0) \leq F [\matho (\iota)] (0)$.
Hence, if $A$ is not the principal formula of ($\bigwedge$), then the
 claim follows immediately from IH.
Suppose that 
$A$ is the principal formula of ($\bigwedge$).
Then $A \simeq \AND{\iota \in J} A_\iota \in \Gamma$ and 
$\Gamma' \equiv \Gamma, A, A_\iota$.
Let $\alpha \leq \xi \leq F (0)$.
Then IH yields 
$f[N (\iota)], F [\matho (\iota)] 
 \vdash^{\alpha_\iota}_{\rho} \Gamma, A^\xi, A_\iota$.
If $A \not\equiv \neg \itP{\ofA}{\Omega} t$, then another application of IH and an
 application of ($\bigwedge$) yield
$f, F \vdash^\alpha_\rho \Gamma, A^\xi$.
If 
$A \equiv \neg \itP{\ofA}{\Omega} t \simeq 
 \AND{\mu < \Omega} \neg \ofA (\itP{\ofA}{\mu}, t)$, then an application of
 ($\bigwedge$) with $\mu \leq \xi \leq F (0) < \Omega$ yields  
$f, F \vdash^\alpha_\rho \Gamma, \neg \itP{\ofA}{\xi}$.

\textsc{Case.}
The last rule is ($\Clrule$):
If $A$ is not the principal formula, then the claim again follows from
 IH.
Let us consider the case that $A$ is the principal formula of the last
 rule ($\Clrule$) with a premise 
$f, F \vdash^{\alpha_0}_{\rho} \Gamma, \itP{\ofA}{\Omega} t, 
 \ofA (\itP{\ofA}{\Omega}, t)$
for some $\alpha_0 < \alpha$ where $A \equiv \itP{\ofA}{\Omega} t$.
Let $\alpha \leq \xi \leq F (0)$.
An application of IH yields
$f, F \vdash^{\alpha_0}_{\rho} \Gamma, 
 \itP{\ofA}{\xi}, \ofA (\itP{\ofA}{\Omega}, t)$.
Another application of IH yields
$f, F \vdash^{\alpha_0}_{\rho} \Gamma, 
 \itP{\ofA}{\xi}, \ofA (\itP{\ofA}{\alpha_0}, t)$.
Let us observe that 
$\matho (\alpha_0) = \alpha_0 < \alpha$,
$N (\alpha_0) \leq f (0)$, and 
$\matho (\alpha_0) = \alpha_0 < \alpha \leq F (0)$.
Hence we can apply ($\bigvee$) with $\alpha_0 < \alpha \leq \xi$,
 concluding $f, F \vdash^\alpha_\rho \Gamma, \itP{\ofA}{\xi}$.
\qed
\end{proof}

We will write 
$f, F \cvdash{\alpha} \Gamma$ instead of
$f, F \vdash^\alpha_{\alpha} \Gamma$.

\begin{lemma}[Impredicative Cut-elimination]
\label{lem:ICE}
\ \\
If $f, F \vdash^{\alpha}_{\Omega +1} \Gamma$, then
$f^{F^{\alpha} (0) +1}, F^{\alpha +1}
 \cvdash{F^\alpha (0)} \Gamma$.
\end{lemma}

\begin{proof}
By induction on $\alpha$.
It is easy to check that 
$f(0) \leq f^{F^\alpha (0) +1} (0)$ and 
$F (0) \leq F^\alpha (0)$.
It also holds that 
$K_{\Omega} F^\alpha (0) = \{ F^\alpha (0) \} < F^{\alpha+1} (0)$.
Further, 
\begin{eqnarray*}
N (F^{\alpha +1} (0)) = N (F(0)) + N (\alpha) +1 &\leq& 
 f(0) + f (0) +1 \\
&\leq& f(f(0)) +1 \\
&\leq& 
 f^{F^\alpha (0) +1} (0).
\end{eqnarray*}
And hence 
$N(F^\alpha (0)) < N(F^{\alpha +1} (0)) \leq f^{F^\alpha (0) +1} (0)$ 
in particular.
Let ($\mathcal J$) denote the last rule that forms
$f, F \vdash^{\alpha}_{\Omega +1} \Gamma$.

\textsc{Case.}
($\mathcal J$) is ($\mathsf{Cut}$) with a cut formula $C$: In this case ($\mathcal J$) has two premises
$f, F \vdash^{\alpha_0}_{\Omega +1} \Gamma, C$ and
$f, F \vdash^{\alpha_0}_{\Omega +1} \Gamma, \neg C$ for some 
$\alpha_0 < \alpha$.
IH yields that
\begin{eqnarray}
f^{F^{\alpha_0} (0) +1}, F^{\alpha_0 +1} 
&\cvdash{F^{\alpha_0} (0)}& \Gamma, C, 
\label{e:l:ICE:1} \\
f^{F^{\alpha_0} (0) +1}, F^{\alpha_0 +1} 
&\cvdash{F^{\alpha_0} (0)}& \Gamma, \neg C.
\label{e:l:ICE:2}
\end{eqnarray}
Let us observe that
$F^{\alpha_0} (0) < F^{\alpha} (0)$ since
$K_{\Omega} \alpha_0 < F (0) \leq F^{\alpha} (0)$.
Similarly 
$F^{\alpha_0 +1} (0) < F^{\alpha +1} (0)$ holds.
Further 
\begin{eqnarray*}
N (F^{\alpha_0} (0) +1) &=& N (F (0)) + N (\alpha_0) +1 \\
&\leq&
N (F^\alpha (0)) + f(0) +1
\qquad \text{ since } N (\alpha_0) \leq f(0), \\
&\leq& f(N (F^\alpha (0) +1))) =
 f [N (F^\alpha (0) +1)] (0).
\end{eqnarray*}
Hence $f^{F^{\alpha_0} (0) +1} (0) < f^{F^\alpha (0) +1} (0)$.

\textsc{Subcase}.
$\rk (C) < \Omega$.
By Lemma \ref{lem:kPi}.\ref{lem:kPi:3} 
$\rk (C) = \rk (\neg C) \leq 
 \omega \cdot (\max \kPO (\neg C)) + \lh (\neg C) < F(0)
$
since
$\kPO (\neg C) \subseteq \kO (C) < F(0)$.
Hence
$\rk (C) < F(0) \leq F^\alpha (0)$.
This together with the two sequents (\ref{e:l:ICE:1}) and
 (\ref{e:l:ICE:2}) allows us to deduce other two sequents 
$f^{F^{\alpha} (0) +1}, F^{\alpha +1} 
 \cvdash{F^{\alpha_0} (0)} \Gamma, C$ and
$f^{F^{\alpha} (0) +1}, F^{\alpha +1} 
\cvdash{F^{\alpha_0} (0)} \Gamma, \neg C$.
We can apply ($\mathsf{Cut}$) to these two sequents, concluding
$f^{F^{\alpha} (0) +1}, F^{\alpha +1} 
\cvdash{F^{\alpha} (0)} \Gamma$.

\textsc{Subcase.}
$\rk (C) = \Omega$.
In this case $C \equiv \itP{\ofA}{\Omega} t$ by Lemma \ref{lem:kPi}.\ref{lem:kPi:4}.
Let us observe the following.
\begin{enumerate}
\item
$N (F^{\alpha_0} (0)) = N (F(0)) + N (\alpha_0) \leq
 f(0) + f(0) \leq f(f(0)) \leq f^{F^{\alpha_0} (0)}$.
\item
$K_{\Omega} F^{\alpha_0} (0) = \{ F^{\alpha_0} (0) \} <
 F^{\alpha_0 +1} (0)$.
\end{enumerate}
Applying Boundedness lemma (Lemma \ref{lem:reflection}) to the sequent 
(\ref{e:l:ICE:1}) yields the sequent
$f^{F^{\alpha_0} (0)}, F^{\alpha_0 +1} 
 \cvdash{F^{\alpha_0} (0)}
 \Gamma, \itP{\ofA}{F^{\alpha_0} (0)}$.
As in the previous subcase this induces the sequent 
\begin{equation}
f^{F^{\alpha} (0) +1}, F^{\alpha +1} 
 \cvdash{F^{\alpha_0} (0)}
 \Gamma, \itP{\ofA}{F^{\alpha_0} (0)}.
\label{e:l:ICE:3} 
\end{equation}
On the other hand applying Inversion lemma (Lemma \ref{lem:inversion})
 to the sequent (\ref{e:l:ICE:2}) yields the sequent
\begin{equation*}
f^{F^{\alpha_0} (0) +1} [N (F^{\alpha_0} (0))], 
 F^{\alpha_0 +1} [F^{\alpha_0} (0)]
\cvdash{F^{\alpha_0} (0)}
\Gamma, \neg \itP{\ofA}{F^{\alpha_0} (0)}.
\end{equation*}
By Property 1 we can see that
$f^{F^{\alpha_0} (0) +1} [N (F^{\alpha_0} (0))] (0) \leq
 f^{F^{\alpha_0} (0) +1} (f^{F^{\alpha_0} (0)} (0))$ $\leq
 f^{F^{\alpha} (0) +1} (0)
$ 
and
$F^{\alpha_0 +1} [F^{\alpha_0} (0)] (0) \leq 
 F^{\alpha +1} (0)$. 
These observations induce the sequent
\begin{equation}
f^{F^{\alpha} (0) +1}, F^{\alpha +1}
\cvdash{F^{\alpha_0} (0)}
\Gamma, \neg \itP{\ofA}{F^{\alpha_0} (0)}.
\label{e:l:ICE:4}
\end{equation}
By definition 
$\rk (\itP{\ofA}{F^{\alpha_0} (0)}) =
 \rk (\neg \itP{\ofA}{F^{\alpha_0} (0)}) =
 F^{\alpha_0} (0) < F^{\alpha} (0)$.
Now by an application of ($\mathsf{Cut}$) to the two sequents
 (\ref{e:l:ICE:3}) and  (\ref{e:l:ICE:4}) we can derive the desired sequent
$f^{F^\alpha (0) +1}, F^{\alpha +1}
 \cvdash{F^\alpha (0)} \Gamma$.

\textsc{Case.}
($\mathcal J$) is ($\bigwedge$) with a principal formula
$A \simeq \AND{\iota \in J} A_{\iota} \in \Gamma$:
In this case 
$\forall \iota \in J$, $\exists \alpha_\iota < \alpha$ s.t.
$f [N (\iota)], F [\matho (\iota)] 
 \vdash^{\alpha_\iota}_{\Omega +1} \Gamma, A_\iota$.
IH yields the sequent 
\begin{equation*}
f [N (\iota)]^{F [\matho (\iota)]^{\alpha_\iota} (0) +1}, 
F [\matho (\iota)]^{\alpha_\iota +1} 
 \cvdash{F [\matho (\iota)]^{\alpha_\iota} (0)}
 \Gamma, A_\iota
\end{equation*}
for all $\iota \in J$.
In the same way as we showed the claim in the proof of Lemma \ref{l:CE1}
(p. \pageref{c:CE1}), one can show that
for all $\iota \in J$
\begin{eqnarray*}
 f [N (\iota)]^{F [\matho (\iota)]^{\alpha_\iota} (0) +1} (0) 
 &\leq&
 f^{F^\alpha (0) +1} [N (\iota)] (0), \\
 F [\matho (\iota)]^{\alpha_\iota +1} (0) &\leq&
 F^{\alpha +1} [\matho (\iota)] (0).
\end{eqnarray*}
These enable us to deduce the sequent
\begin{equation*}
f^{F^{\alpha} (0) +1} [N (\iota)], F^{\alpha +1} [\matho (\iota)] 
\vdash^{F [\matho (\iota)]^{\alpha_\iota} (0)}_{F^{\alpha} (0)}
 \Gamma, A_\iota
\end{equation*}
for all $\iota \in J$.
Since 
$F [\matho (\iota)]^{\alpha_\iota} (0) < F^\alpha (0)$
for all $\iota \in J$, we can apply
($\bigwedge$) to this sequent, concluding
$f^{F^{\alpha} (0) +1}, F^{\alpha +1}
 \cvdash{F^{\alpha} (0)} \Gamma$.
\qed
\end{proof}

\begin{lemma}[Witnessing]
\label{lem:witness}
For each $j < l$ let $B_j (x)$ be a 
$\Delta^0_0$-$\mathcal{L}_{\mathrm{PA}}$-formula such that 
$\mathsf{FV} (B_j (x)) = \{ x \}$.
Let
$\Gamma \equiv 
 \exists x_0 B_0 (x_0), \dots, 
 \exists x_{l-1} B_{l-1} (x_{l-1})$.
If $f, F \vdash^\alpha_0 \Gamma$ for some $\alpha \in \OT$, then
there exists a sequence
$\langle m_0, \dots, m_{l-1} \rangle$ of naturals such that 
$\max \{ m_j \mid j < l \}  \leq f(0)$ and
$
 B_0 (\ul{m_0}) \vee \cdots \vee B_{l-1} (\ul{m_{l-1}})$
is true in the standard model $\mathbb N$ of $\mathrm{PA}$.
\end{lemma}

\begin{proof}
By induction on $\alpha$.
The derivation forming $f, F \vdash^\alpha_0 \Gamma$ contains no 
($\mathsf{Cut}$) rules.
Hence the last inference rule should be ($\bigvee$).
Thus there exist an ordinal $\alpha_0 < \alpha$ and a (closed) term 
$t \in \mathcal{T} (\LID)$ such that 
$N (t) \leq f(0)$ and
$f, F \vdash^{\alpha_0}_0 \Gamma, B_{l-1} (t)$.
By IH there exists a sequence 
$\langle m_0, \dots, m_{l-1} \rangle$ of naturals such that
$\max \{ m_j \mid j < l \}  \leq f(0)$ and
$
 B_0 (\ul{m_0}) \vee \cdots \vee B_{l-1} (\ul{m_{l-1}}) \vee
 B_{l-1} (t)$
is true in $\mathbb N$.
If 
$
 B_0 (\ul{m_0}) \vee \cdots \vee B_{l-1} (\ul{m_{l-1}})
$
is already true in $\mathbb N$, 
then 
$\langle m_0, \dots, m_{l-1} \rangle$ is the desired sequence.
Suppose that
$
 B_0 (\ul{m_0}) \vee \cdots \vee B_{l-1} (\ul{m_{l-1}})
$
is not true in $\mathbb N$. 
Then
$B_{l-1} (t)$ must be true. 
Hence
$B_{l-1} (\ul{\val (t)})$ is also true.
By definition,
$\val (t) =  N (t) \leq f(0)$, and hence
$\langle m_0, \dots, m_{l-2}, \val (t) \rangle$ is the desired sequence.
\qed
\end{proof}

\section{Embedding $\mathbf{ID}_1$ into $\IDomega$}

In this section we embed the theory $\mathbf{ID}_1$ into the infinitary
system $\IDomega$.
Following conventions in the previous section we use the symbol $f$ to
denote a strict increasing function
$f: \mathbb{N} \rightarrow \mathbb{N}$ that enjoys 
the conditions ($f$.\ref{f:1}) and  ($f$.\ref{f:2})
(p. \pageref{f:1}).
Let us recall that the function symbol $\EN \in \mathcal{F}$ denotes 
the function $\EN: \Omega \rightarrow \Omega$ such that
$\EN (\alpha) = \min \{ \xi < \Omega \mid \alpha < \xi \text{ and }
 \xi = \omega^\xi \}$.
It is easy to see that the condition $(\hyp{}{}(\EN))$ holds since
$\EN (\xi) = \varepsilon_0 \leq \EN (0)$
for all $\xi < \EN (0) = \varepsilon_0$.


\begin{lemma}[Tautology lemma]
\label{l:taut}
Let $s, t \in \mathcal{T} (\LID)$,
$\Gamma$ be a sequent of $\Lstar$-sentences, and 
$A(x)$ be an $\Lstar$-formula such that 
$\mathsf{FV} (A) = \{ x \}$.
If $\val (s) = \val (t)$, then
\begin{equation}
 f[n], 
 \EN [\kO (A)] \vdash^{\rk (A) \cdot 2}_0 
 \Gamma, \neg A(s), A(t),
\label{e:l:taut}
\end{equation}
where 
$n:= \max \{ N (\rk (A)), N(\max \kPO (A)), N(\max \kSO (A)) \}$.
\end{lemma}

\begin{proof}
By induction on $\rk (A)$.
Let
$n$ denote the maximal among
$N (\rk (A))$, $N(\max \kPO (A))$ and $N(\max \kSO (A)) \}$.
From Lemma \ref{lem:kPi}.\ref{lem:kPi:3} one can check that the condition
$\hyp{}{}(f[n]; \EN (\kO (A)); \rk (A) \cdot 2)$ holds.
If $\rk (A) =0$, then $A$ is an $\LID$-literal, and hence 
(\ref{e:l:taut})
is an instance of ($\mathsf{Ax}1$).
Suppose that $\rk (A) > 0$. 
Without loss of generality we can assume that
$A \simeq \OR{\iota \in J} A_\iota$.
Let $\iota \in J$.
By Lemma \ref{lem:kPi}.\ref{lem:kPi:5} let us observe that
$N(\rk (A_\iota) \cdot 2) < 2 \{ N (\rk (A)), N(\iota) \} \leq
 f[N (\rk (A))] [N(\iota)] (0) \leq f[n] (0)$
since $2m+1 \leq f(m)$ for all $m$ by the condition ($f$.\ref{f:1}).
Further by Lemma \ref{lem:kPi}.\ref{lem:kPi:1}
$K_{\Omega} (\rk (A_\iota) \cdot 2) \subseteq \kO (A) \cup
 \{ \matho (\iota) \} \leq
 \EN [\kO (A)] [\matho (\iota)]$.
Summing up, we have the condition
\[
 \hyp{}{}(f[n] [N(\iota)]; 
 \EN [\kO (A)] [\matho (\iota)]; \rk (A_\iota) \cdot 2).
\]
Hence by IH we can obtain the sequent
\begin{eqnarray}
 f[n] [N(\iota)], 
 \EN [\kO (A)] [\matho (\iota)]
 \vdash^{\rk (A_\iota) \cdot 2}_0 
 \Gamma, \neg A_\iota (s), A_\iota (t).
\label{e:l:taut:1}
\end{eqnarray}
It is not difficult to see 
$\matho (\iota) \leq \rk (A_\iota) < \rk (A_\iota) \cdot 2 +1$
and
$N (\rk (A_\iota) \cdot 2 +1) = N (\rk (A_\iota) \cdot 2) +1 \leq
 f[N(\rk (A))] [N(\iota)] (0) \leq f[n](0)$.
This allows us to apply
 ($\bigvee$) to the sequent (\ref{e:l:taut:1}) yielding
\[
 f[n] [N (\iota)], 
 \EN [\kO (A)] [\matho (\iota)]
 \vdash^{\rk (A_\iota) \cdot 2 +1}_0
 \Gamma, \neg A_\iota (s), A(t).
\]
We can see that $\rk (A_\iota) \cdot 2 +1 < \rk (A) \cdot 2$,
$N(\max \kPO (A)) \leq f[n](0)$
and
$\kPO (A) < \EN [\kO (A)]$.
Hence we can apply ($\bigwedge$) concluding
(\ref{e:l:taut}).
\qed
\end{proof}

\begin{lemma}
\label{l:PL}
Let $B_j$ be an $\LID$-sentence for each $j =0, \dots, l$.
Suppose that 
$(\neg B_0) \vee \cdots \vee (\neg B_{l-1}) \vee B_l$
is a logical consequence in the first order predicate logic
 with equality.
Then there exists a natural
$k < \omega$ such that
$f[m+k], \EN \vdash^{\Omega \cdot 2 +k}_0 \{ B_j \mid 0 \leq j \leq l \}$, 
where
$m = \max \{ N(\rk (B_j)) \mid j =0, 1, \dots, l \}$.
\end{lemma}

\begin{proof}
Let $B_j$ be an $\LID$-sentence for each $j =0, \dots, l-1$ and 
suppose that $B_0 \vee \cdots \vee B_{l-1}$
is a logical consequence in the first order predicate logic
 with equality.
Then we can find a cut-free proof of the sequent 
$\{ B_j \mid 0 \leq j \leq l-1 \}$ in an {\bf LK}-style sequent calculus.
More precisely we can find a cut-free proof $P$ of
$\{ B_j \mid 0 \leq j \leq l-1 \}$
in the sequent calculus $\mathbf{G3_m}$.
(See the book \cite{TS2000} of Troelstra and Schwichtenberg for the
 definition.)
Let $h$ denote the tree height of the cut-free proof $P$.
Then by induction on $h$ one can find a witnessing natural 
$k$ such that
$f[m+k], F \vdash^{\alpha}_0 \{ B_j \mid 0 \leq j \leq l-1 \}$
for all $\alpha \geq \Omega +k$.
In case $h=0$ Tautology lemma (Lemma \ref{l:taut}) can be applied since
for any $\LID$-sentence $A$,
$\rk (A) \in \omega \cup \{ \Omega + k \mid k < \omega \}$ and 
$\kPi (A) \cup \kSigma (A) = \kPS (A) \subseteq \{ 0, \Omega \}$,
and hence
$\kO (A) = \{ 0 \}$ and
$\max \{ N(\max \kPO (A)), N(\max \kSO (A)) \} = 0$.
\qed
\end{proof}

\begin{lemma}
\label{l:ind}
Let $m \in \mathbb{N}$ and $A(x)$ be an 
$\LID$-formula such that $\mathsf{FV} (A(x)) = \{ x \}$.
Then for any $t \in \mathcal{T} (\LID)$ and for any sequent
 $\Gamma$ of $\LID$-sentences, if $\val (t) = m$, then 
\begin{equation}
 f [N (\rk (A)) +m], \EN \vdash^{(\rk (A) + m) \cdot 2}_0 
 \Gamma, \neg A(0), \neg \forall x (A(x) \rightarrow A(S(x))), A(t).
\label{e:l:ind:0}
\end{equation}
\end{lemma}

\begin{proof}
By induction on $m$.
The base case $\val (t) = m =0$ follows from Tautology lemma (Lemma \ref{l:taut}).
For the induction step suppose $\val (t) = m+1$.
Fix a sequent $\Gamma$ of $\LID$-sentences.
Then (\ref{e:l:ind:0}) holds by IH.
On the other hand again by Tautology lemma, 
\begin{equation}
f[N (\rk (A))], \EN \vdash^{\rk (A) \cdot 2}_0
\Gamma, \neg A(0), \exists x (A(x) \wedge \neg A(S(x))),
A(\ul{m}), \neg A(\ul{m}).
\label{e:l:ind:1}
\end{equation}
An application of ($\bigwedge$) to the two sequents (\ref{e:l:ind:0})
 and (\ref{e:l:ind:1}) yields
\begin{equation*} 
f[N (\alpha_m)], \EN \vdash^{\alpha_m \cdot 2 +1}_0
\Gamma, \neg A(0), \exists x (A(x) \wedge \neg A(S(x))),
A(t), A(\ul{m}) \wedge \neg A(\ul{m}),
\end{equation*}
where $\alpha_m := \rk (A)) +m$.
The final application of ($\bigvee$) yields
\begin{equation*} 
f[N (\rk (A)) +m+1], F \vdash^{(\rk (A) +m+1) \cdot 2}_0
\Gamma, \neg A(0), \exists x (A(x) \wedge \neg A(S(x))),
A(t).
\end{equation*}
\qed
\end{proof}

\begin{lemma} 
\label{l:GIL}
Let $\xi \leq \Omega$,
$F(x)$ be an $\LID$-formula such that
$\mathsf{FV} (F(x)) = \{ x \}$ and $B(X)$ be an $X$-positive
 $\mathcal{L}_{\mathrm{PA}} (X)$-formula such that
$\mathsf{FV} (B) = \emptyset$.
Then
\[
 f[N (\sigma + \alpha +1)], 
 \EN [K_{\Omega} \xi] \vdash^{(\sigma + \alpha +1) \cdot 2}_0
 \Gamma, \neg \forall x (\ofA (F, x) \rightarrow F(x)), 
 \neg B(\itP{\ofA}{\xi}), B(F),
\]
where
$\sigma := \rk (F)$ and
$\alpha := \rk (B(\itP{\ofA}{\xi}))$.
\end{lemma}

\begin{proof}
By main induction on $\xi$ and side induction on 
$\rk (B(\itP{\ofA}{\xi}))$.
Let
$\Cl_{\ofA} (F) :\equiv 
 \neg \forall x (\ofA (F, x) \rightarrow F(x)) \equiv
 \exists x (\ofA (F, x) \wedge \neg F(x))$.
The argument splits into several cases depending on the shape of
the formula $B(X)$.

\textsc{Case.}
$B(X)$ is an $\LPA$-literal:
In this case $B$ does not contain the set free variable $X$, and hence
Tautology lemma (Lemma \ref{l:taut}) can be applied.

\textsc{Case.}
$B \equiv X(t)$ for some $t \in \mathcal{T}(\LID)$:
In this case
$\neg B(\itP{\ofA}{\xi}) \equiv \neg \itP{\ofA}{\xi} t \equiv 
 \AND{\eta < \xi} \neg \ofA (\itP{\ofA}{\eta}, t)$.
Let $\eta < \xi$.
Then by MIH 
\begin{equation*}
f[N(\sigma + \alpha_\eta +1)], \EN [K_{\Omega} \eta] 
\vdash^{(\sigma + \alpha_\eta +1) \cdot 2}_0
\Gamma, \Cl_{\ofA} (F), \neg \ofA (\itP{\ofA}{\eta}, t),
\ofA (F, t), F(t)
\end{equation*}
where
$\alpha_\eta := \rk (\ofA (\itP{\ofA}{\eta}, t))$.
We note that $\eta < \xi \leq \Omega$ and hence
$K_{\Omega} \eta = \{ \eta \} = \{ \matho (\eta) \}$.
Hence this yields the sequent
\begin{equation*}
f[N(\sigma + \alpha)] [N (\eta)], 
\EN [\matho (\eta)] 
\vdash^{(\sigma + \alpha_\eta +1) \cdot 2}_0
\Gamma, \Cl_{\ofA} (F), \neg \ofA (\itP{\ofA}{\eta}, t),
\ofA (F, t), F(t).
\end{equation*}
An application of $(\AND{})$ yields the sequent
\begin{equation}
f[N(\sigma + \alpha)], \EN [K_\Omega \xi]
\vdash^{(\sigma + \alpha) \cdot 2}_0
\Gamma, \Cl_{\ofA} (F), \neg \itP{\ofA}{\xi} t,
\ofA (F, t), F(t).
\label{e:l:GNL:1}
\end{equation}
On the other hand by Tautology lemma (Lemma \ref{l:taut}),
\begin{equation}
f[N (\sigma + \alpha)], \EN [K_\Omega \xi] 
\vdash^{\rk (F) \cdot 2}_0
\Gamma, \Cl_{\ofA} (F), \neg \itP{\ofA}{\xi} t,
\neg F(t), F(t).
\label{e:l:GNL:2}
\end{equation}
Another application of $(\AND{})$ to the two sequents (\ref{e:l:GNL:1}) and
 (\ref{e:l:GNL:1}) yields the sequent
\begin{equation*}
f[N(\sigma + \alpha +1)], \EN [K_\Omega \xi]
\vdash^{(\sigma + \alpha) \cdot 2 +1}_0
\Gamma, \Cl_{\ofA} (F), \neg \itP{\ofA}{\xi} t,
\ofA (F, t) \wedge \neg F(t), F(t).
\end{equation*}
An application of $(\OR{})$ allows us to conclude
\begin{equation*}
f[N(\sigma + \alpha +1)], \EN [K_\Omega \xi]
\vdash^{(\sigma + \alpha +1) \cdot 2}_0
\Gamma, \Cl_{\ofA} (F), \neg \itP{\ofA}{\xi} t,
F(t).
\end{equation*}

\textsc{Case.}
$B(X) \equiv \forall y B_0 (X, y)$ for some $\LPA$-formula $B_0 (X, y)$:
Let 
$\alpha_0$ denote the ordinal 
$\rk (B_0 (\itP{\ofA}{\xi}, \ul{0}))$.
Then $\alpha = \alpha_0 +1$.
By the definition of the rank function $\rk$, 
$\alpha_0 = \rk (B_0 (\itP{\ofA}{\xi}, t))$
for all $t \in \mathcal{T} (\LID)$. 
Fix a closed term $t \in \mathcal{T} (\LID)$. 
Then from SIH we have the sequent
\begin{equation*}
f[N(\sigma + \alpha +1)], \EN [K_{\Omega} \xi]
\vdash^{(\sigma + \alpha) \cdot 2}_0
\Gamma, \Cl_{\ofA} (F), \neg B_0 (\itP{\ofA}{\xi}, t),
B_0 (\itP{\ofA}{\xi}, t).
\end{equation*}
An application of $(\OR{})$ yields the sequent
\begin{equation*}
f[N(\sigma + \alpha +1)], \EN [K_{\Omega} \xi]
\vdash^{(\sigma + \alpha) \cdot 2 +1}_0
\Gamma, \Cl_{\ofA} (F), \neg \forall y B_0 (\itP{\ofA}{\xi}, y),
B_0 (\itP{\ofA}{\xi}, t).
\end{equation*}
And an application of $(\AND{})$ allows us to conclude.

The other cases can be treated in similar ways.
\qed
\end{proof}

\begin{lemma}
\label{l:ID}
\begin{enumerate}
\item $f[N(\rk (\ofA (\itP{\ofA}{\Omega}, \ul{0})) +1], 
       \EN \vdash^{\Omega \cdot 2 + \omega}_0 
       \forall x (\ofA (\itP{\ofA}{\Omega}, x) \rightarrow 
                  \itP{\ofA}{\Omega} x)$.
\label{l:ID:1}
\item $f[3+l], \EN \vdash^{\Omega \cdot 2 + \omega}_0
       \forall \vec{y} 
       [\forall x \{ \ofA (F(\cdot, \vec{y}), x) \rightarrow F(x, \vec{y})
                  \} \rightarrow
        \forall x \{ \itP{\ofA}{\Omega} x \rightarrow F(x, \vec{y}) \}
       ]$, 
       where $\vec{y} = y_0, \dots, y_{l-1}$.
\label{l:ID:2}
\end{enumerate}
\end{lemma}

\begin{proof}
{\sc Property} \ref{l:ID:1}.
Let $\alpha = \rk (\ofA (\itP{\ofA}{\Omega}, \ul{0})$ and
$t \in \mathcal{T} (\LID)$.
By the definition of $\rk$ we can find a natural $k < \omega$ such that
$\alpha = \rk (\ofA (\itP{\ofA}{\Omega}, t) =
 \Omega + k$.
This implies 
$\kPS (\ofA (\itP{\ofA}{\Omega}, t)) = \{ 0, \Omega \}$
and hence
$\kO (\ofA (\itP{\ofA}{\Omega}, t)) = \{ 0 \} < \EN (0)$.
By Tautology lemma (Lemma \ref{l:taut}),
\begin{equation*}
f[N(\alpha)], \EN \vdash^{\alpha \cdot 2}_0
\itP{\ofA}{\Omega} t,
\neg \ofA (\itP{\ofA}{\Omega}, t), \ofA (\itP{\ofA}{\Omega}, t).
\end{equation*}
Since 
$\Omega < \Omega \cdot 2 + k+1 =  \alpha \cdot 2 +1$,
we can apply the closure rule $(\Clrule)$ obtaining the sequent
\begin{equation*}
f[N(\alpha)], \EN \vdash^{\Omega \cdot 2 +k+1}_0
\neg \ofA (\itP{\ofA}{\Omega}, t), \itP{\ofA}{\Omega} t.
\end{equation*}
An application of $(\AND{})$ followed by an application of 
$(\OR{})$ enables us to conclude
\begin{equation*}
f[N(\alpha)+1], \EN \vdash^{\Omega \cdot 2 + \omega}_0
\forall x (\ofA (\itP{\ofA}{\Omega}, x) \rightarrow \itP{\ofA}{\Omega} x).
\end{equation*}

{\sc Property} \ref{l:ID:2}.
By definition
$\rk (\itP{\ofA}{\Omega}) = \omega \cdot \Omega = \Omega$,
On the other hand
$\rk (F) < \omega$ and hence
$(\rk (F) + \rk (\itP{\ofA}{\Omega}) +1) \cdot 2 = \Omega \cdot 2 +2$.
Let 
$s, \vec{t} = s, t_0, \dots t_{l-1} \in \mathcal{T} (\LID)$.
Then by the previous lemma (Lemma \ref{l:GIL})
\begin{equation*}
f[2], \EN \vdash^{\Omega \cdot 2 +1}_0 
\neg \forall x (\ofA (F(\cdot, \vec{t}), x) \rightarrow F(x, \vec{t})),
\neg \itP{\ofA}{\Omega} t, F(s, \vec{t})
\end{equation*}
since $N(\Omega +1) =2$.
It is not difficult to see that applications of $(\OR{})$, $(\AND{})$
 and $(\OR{})$ in this order yield the sequent
\begin{equation*}
f[3], \EN \vdash^{\Omega \cdot 2 +5}_0 
\forall x (\ofA (F(\cdot, \vec{t}), x) \rightarrow F(x, \vec{t}))
\rightarrow
\forall x (\itP{\ofA}{\Omega} x \rightarrow F(x, \vec{t}))
\end{equation*}
Finally, $l$-fold application of $(\AND{})$ allows us to conclude.
\qed
\end{proof}

Let us recall that $\suc$ denotes the numerical successor $m \mapsto m+1$.

\begin{theorem}
Let $A \equiv \forall \vec x \exists y B (\vec x, y)$ be a
 $\Pi^0_2$-sentence for a $\Delta^0_0$-formula $B(\vec x, y)$
such that $\mathsf{FV} (B (\vec x, y)) = \{ \vec{x}, y \}$.
If $\mathbf{ID}_1 \vdash A$, then we can an ordinal term 
$\alpha \in \OT \seg \Omega$ built up without the Veblen function symbol 
$\varphi$ such that for all 
$\vec m = m_0, \dots, m_{l-1} \in \mathbb N$ there exists 
$n \leq \suc^\alpha (m_0 + \cdots + m_{l-1})$ such that
$B (\vec m, n)$ is true in the standard model $\mathbb N$ of $\mathrm{PA}$.
\end{theorem}

\begin{proof}
Assume $\mathbf{ID}_1 \vdash A$.
Then there exist $\mathbf{ID}_1$-axioms $A_1, \dots, A_{k}$ such that
$(\neg A_1) \vee \cdots (\neg A_{k}) \vee A$
is a logical consequence in the first order predicate logic with
 equality.
Hence by Lemma \ref{l:PL},
\begin{equation*}
f[c_0], \EN \vdash^{\Omega \cdot 3}_0 
\neg A_1, \dots, \neg A_{k}, A
\end{equation*}
for some constant $c_0 < \omega$ depending on
$N(\rk (A_1)), \dots, N(\rk (A_k))$, $N(\rk (A))$ 
and depending also on the tree height of a cut-free $\mathbf{LK}$-derivation
 of the sequent
$\neg A_1, \dots, \neg A_{l}, A$.
By Lemma \ref{l:ind} and \ref{l:ID}, for each $j=1, \dots, k$,
there exists a constant $c_j$ depending on $\rk (A_j)$ such that
$f[c_j], \EN \vdash^{\Omega \cdot 2 + \omega}_0 A_j$.
Hence $k$-fold application of $(\mathsf{Cut})$ yields
$f[c], \EN \vdash^{\Omega \cdot 3}_{\Omega + d +1} A$,
where
$c := \max (\{ k \} \cup \{ c_j \mid j \leq k \} \cup 
            \{ \lh (A_j) \mid 1 \leq j \leq k \}
           )$ 
and
$d := \max (\{ \Omega, \rk (A_1), \dots, \rk (A_k) \})$.

For each $n \in \mathbb N$ and $\alpha \in \OT$ let us define ordinal 
$\Omega_n (\alpha)$ and $\gamma_n$ by
\[
 \begin{array}{rclrcl}
 \Omega_0 (\alpha) &=& \alpha, &
 \gamma_0 &=& \Omega \cdot 3, \\
 \Omega_{n+1} (\alpha) &=& \Omega^{\Omega_n (\alpha)}, \quad &
 \gamma_{n+1} &=& \EN^{\gamma_n} (0) +1.
 \end{array}
\]

Then $d$-fold iteration of Cut-reduction lemma (Lemma \ref{lem:cut-red})
 yields the sequent
$f[c]^{\gamma_d}, \EN \vdash^{\Omega_d (\Omega \cdot 3)}_{\Omega +1} A$.
Hence Impredicative cut-elimination lemma (Lemma \ref{lem:ICE}) yields 
\[
 (f[c]^{\gamma_d})^{\EN^{\Omega_d (\Omega \cdot 3)} (0)},
 \EN^{\Omega_d (\Omega \cdot 3) +1} 
 \cvdash{\EN^{\Omega_d (\Omega \cdot 3)} (0)} A.
\]
Let 
$F := \EN^{\Omega_d (\Omega \cdot 3) +1}$ 
and 
$\beta := \EN^{\Omega_d (\Omega \cdot 3)} (0)$.
Then
$(f[c]^{\gamma_d})^\beta, F
 \vdash^{\beta}_{\omega^\beta} A$ holds.
It is not difficult to check that 
$\beta < \Omega$, $N (\beta) \leq (f[c]^{\gamma_d})^\beta$ and 
$K_{\Omega} \beta < F(0)$.
Hence Predicative cut-elimination lemma (Lemma \ref{lem:PCE}) yields the sequent
\[
 (f[c]^{\gamma_d})^{F^{\Omega \cdot \beta + \beta \cdot 2} (0) +1}
 F \vdash^{\varphi \beta \beta}_0 A.
\]
Now let $f$ denote $\suc^\omega$.
By Example \ref{ex:1}.\ref{ex:1:4} one can check that the conditions 
($\suc^\omega$.\ref{f:1}) and ($\suc^\omega$.\ref{f:2}) hold.
From Example \ref{ex:1} one will also see that
$\suc^\omega [c] (m) \leq \suc^\omega (\suc^c (m)) \leq 
 \suc^{\omega + c +1} (m)$
for all $m$.
By these we have the inequality
\[
 (\suc [c]^{\gamma_d})^{F^{\Omega \cdot \beta + \beta \cdot 2} (0) +1} (0) 
 \leq
 ((\suc^{\omega +c+1})^{\gamma_d})^{F^{\Omega \cdot \beta + \beta \cdot 2} (0) +1} (0).
\]
Thanks to Lemma \ref{lem:(f^a)^b}  
we can find an ordinal $\alpha \in \OT \seg \Omega$ built up without
the Veblen function symbol $\varphi$ such that 
\[
 ((\suc^{\omega + c+1})^{\gamma_d})^{F^{\Omega \cdot \beta + \beta \cdot 2} (0) +1} (0)
 \leq \suc^\alpha (0).
\]
This together with ($l$-fold application of) Inversion lemma (Lemma
 \ref{lem:inversion}) yields the sequent
\[
 \suc^\alpha [m_0] \cdots [m_{l-1}], F 
 \vdash^{\varphi \beta \beta}_0 
 \exists y B (\ul{\vec m}, y),
\]
where $\vec m = m_0, \dots, m_{l-1}$.
By Witnessing lemma (Lemma \ref{lem:witness}) 
we can find a natural 
$n \leq 
 \suc^\alpha [m_0] \cdots [m_{l-1}] (0) =
 \suc^\alpha (m_0 + \cdots + m_{l-1})$ 
such that
$B (\vec m, n)$ is true in the standard model $\mathbb N$ of $\mathrm{PA}$.
\qed
\end{proof}

We say a function $f$ is {\em elementary} (in another function $g$) if
$f$ is definable explicitly from the successor $\suc$, projection, zero
$0$, addition $+$, multiplication $\cdot$, cut-off subtraction $\minus$
(and $g$), using composition, bounded sums and bounded products,
c.f. Rose \cite[page 3]{Rose}.

\begin{corollary}
\label{c:main:onlyif}
Every function provably computable in $\mathbf{ID}_1$ is elementary in
$\{ \suc^\alpha \mid \alpha \in \OT \seg \Omega \}$.
\end{corollary}

\section{A recursive ordinal notation system $\OTO$}

In order to obtain a precise characterisation of the provably computable
functions of $\mathbf{ID}_1$, we introduce a {\em recursive} ordinal notation system
$\langle \OTO, < \rangle$.
Essentially $\OTO$ is a subsystem of $\OT$.

\begin{definition}
We define three sets 
$\SC \subseteq \AI \subseteq \OTO$ of ordinal terms simultaneously. 
Let $0$, $\Omega$, $\Suc$, and $+$ be distinct symbols.
\begin{enumerate}
\item $0 \in \OTO$ and $\Omega \in \SC$.
\item If $\alpha \in \OT \seg \Omega$, then 
      $\Suc (\alpha) \in \OTO$. 
\item If $\{ \alpha_1, \dots, \alpha_l \} \subseteq \AI$ and
      $\alpha_1 \geq \cdots \geq \alpha_l$, then 
      $\alpha_1 + \cdots + \alpha_l \in \OTO$.
\item If $\alpha \in \OTO$, then 
      $\omega^\alpha \in \AI$.
\item If $\alpha \in \OTO$ and $\xi \in \OTO \seg \Omega$, then
      $\Omega^\alpha \cdot \xi \in \AI$.
\item If $\alpha \in \OTO$ and $\xi \in \OTO \seg \Omega$,
      then $\Suc^\alpha (\xi) \in \SC$.
\end{enumerate}
\end{definition}

The relation $<$ on $\OTO$ is defined in the obvious way.
One will see that $\OTO$ is indeed a recursive ordinal notation system.
Let us define the norm $N(\omega^\alpha)$ of $\omega^\alpha$ in the most
natural way, i.e.,
$N(\omega^\alpha) = N(\alpha) +1$.

\begin{lemma}
Let $\alpha$ denote an ordinal term built up in $\OT$ without the
 Veblen function symbol $\varphi$.
Then there exists an ordinal term $\alpha' \in \OTO$ such that
$\alpha \leq \alpha'$
and
$N(\alpha) \leq N(\alpha')$.
\end{lemma}

\begin{proof}
By induction over the term construction of $\alpha \in \OT$.
In the base case let us observe that
$\EN (\alpha) \leq \Suc^1 (\alpha)$
for all $\alpha < \Omega$ and that
$N(\EN (\alpha)) = N(\alpha) +1 < N(\Suc (\alpha)) + 1 =
 N(\Suc^1 (\alpha))$.
In the induction case we employ Lemma \ref{lem:F^a^b}.
\qed
\end{proof}
 
\begin{lemma}
\label{l:OTO}
For any ordinal term $\alpha \in \OT$ built up without the Veblen
 function symbol $\varphi$ there exists an ordinal term
$\alpha' \in \OTO$ such that
$\suc^\alpha (m) \leq \suc^{\alpha'} (m)$
for all $m$.
\end{lemma}

\begin{corollary}
\label{c:main}
A function is provably computable in $\mathbf{ID}_1$ if and only if it is
 elementary in 
$\{ \suc^\alpha \mid \alpha \in \OTO \seg \Omega \}$.
\end{corollary}

The ``only if'' direction follows from Corollary \ref{c:main:onlyif} and
Lemma \ref{l:OTO}.
The ``if'' direction can be seen as follows.
One can show that for each $\alpha \in \OTO \seg \Omega$ the system
$\mathbf{ID}_1$ proves that the initial segment 
$\langle \OTO \seg \alpha, < \rangle$ of $\langle \OTO, < \rangle$ is a
well-ordering.
For the full proof, we kindly refer the readers to, e.g., Pohlers \cite[\S 29]{Poh98}. 
From this one can show that for each $\alpha \in \OTO \seg \Omega$
the function $\suc^\alpha$ is provably computable in $\mathbf{ID}_1$,
and hence the assertion.

\section{Conclusion}

In this technical report we introduce a new approach to provably computable
functions, providing a simplified characterisation of those of the system
$\mathbf{ID}_1$ of non-iterated inductive definitions.
The simplification is made possible due to the method of
operator-controlled derivations that was originally introduced by
Wilfried Buchholz \cite{Buch92}.
An new idea in this report is to combine the ordinal operators from
\cite{Buch92} with the number-theoretic operators from 
\cite{weier06}, c.f. Definition \ref{d:OCD}.
Ordinal operators contain information much enough to analyse
$\Pi^1_1$-consequences of the controlled derivations.
In contrast, number-theoretic operators contain information  much enough
to analyse those $\Pi^0_2$-consequences.
It is not difficult to generalise this approach to the system $\mathbf{ID}_n$ of
$n$-fold iterated inductive definitions.
Then it is natural to ask whether this approach can be extended to
stronger systems like fragments of Kripke-Platek set theories.
Extension to strong fragments, e.g., the fragment $\mathrm{KPM}$ for
recursively Mahlo universes or the fragment
$\mathrm{KP \Pi_3}$ for $\Pi_3$-reflecting universes, is still a challenge. 


\begin{thebibliography}{10}

\bibitem{arai03}
T.~Arai.
\newblock {Proof Theory for Theories of Ordinals -- I: Recursively Mahlo
  Ordinals}.
\newblock {\em Annals of Pure and Applied Logic}, 122(1--3):1--85, 2003.

\bibitem{arai04}
T.~Arai.
\newblock {Proof Theory for Theories of Ordinals -- II: $\Pi_3$-reflection}.
\newblock {\em Annals of Pure and Applied Logic}, 129(1--3):39--92, 2004.

\bibitem{Bar75}
J.~Barwise.
\newblock {\em {Admissible Sets and Structures. An Approach to Definability
  Theory}}.
\newblock Perspectives in Mathematical Logic. Springer-Verlag, Berlin-New York,
  1975.

\bibitem{BW96}
B.~Blankertz and A.~Weiermann.
\newblock {How to Characterize Provably Total Functions by the Buchholz
  Operator Method}.
\newblock {\em Lecture Notes in Logic}, 6:205--213, 1996.

\bibitem{weier99}
B.~Blankertz and A.~Weiermann.
\newblock {A Uniform Approach for Characterizing the Provably Total
  Number-Theoretic Functions of KPM and (Some of) its Subsystems}.
\newblock {\em Studia Logica}, 62:399--427, 1999.

\bibitem{Buch92}
W.~Buchholz.
\newblock {A Simplified Version of Local Predicativity}.
\newblock In P.~Aczel, H.~Simmons, and S.~Wainer, editors, {\em Proof Theory},
  pages 115--148. Cambridge University Press, Cambridge, 1992.

\bibitem{buch01}
W.~Buchholz.
\newblock {Finitary Treatment of Operator Controlled Derivations}.
\newblock {\em Mathematical Logic Quarterly}, 47(3):363--396, 2001.

\bibitem{BCW94}
W.~Buchholz, E.~A. Cichon, and A.~Weiermann.
\newblock {A Uniform Approach to Fundamental Sequences and Hierarchies}.
\newblock {\em Mathematical Logic Quarterly}, 40(2):273--286, 1994.

\bibitem{FW98}
M.~Fairtlough and S.~S. Wainer.
\newblock {Hierarchy of Provably Recursive Functions}.
\newblock In S.~R. Buss, editor, {\em Handbook of Proof Theory}, pages
  149--207. North Holland, Amsterdam, 1998.

\bibitem{Mich06}
M.~Michelbrink.
\newblock {A Buchholz Derivation System for the Ordinal Analysis of
  $\mathrm{KP} + \Pi_3$-reflection}.
\newblock {\em Journal of Symbolic Logic}, 71(4):1237--1283, 2006.

\bibitem{Poh89}
W.~Pohlers.
\newblock {\em {Proof Theory. An Introduction}}, volume 1407 of {\em Lecture
  Notes in Mathematics}.
\newblock Springer, 1989.

\bibitem{Poh98}
W.~Pohlers.
\newblock {Subsystems of Set Theory and Second Order Number Theory}.
\newblock In S.~R. Buss, editor, {\em Handbook of Proof Theory}, pages
  210--335. North Holland, Amsterdam, 1998.

\bibitem{Rath91}
M.~Rathjen.
\newblock {Proof-theoretic Analysis of KPM}.
\newblock {\em Archive for Mathematical Logic}, 30(5--6):377--403, 1991.

\bibitem{Rath94}
M.~Rathjen.
\newblock {Proof Theory of Reflection}.
\newblock {\em Annals of Pure and Applied Logic}, 68(2):181--224, 1994.

\bibitem{Rose}
H.~E. Rose.
\newblock {\em {Subrecursion: Functions and Hierarchies}}.
\newblock Clarendon Press, Oxford, 1984.

\bibitem{takeuti87}
G.~Takeuti.
\newblock {\em Proof Theory}.
\newblock North-Holland, Amsterdam, 2nd edition, 1987.

\bibitem{TS2000}
A.~S. Troelstra and H.~Schwichtenberg.
\newblock {\em Basic {P}roof {T}heory}.
\newblock Cambridge University Press, Cambridge, 2nd edition, 2000.

\bibitem{weier96}
A.~Weiermann.
\newblock {How to Characterize Provably Total Functions by Local
  Predicativity}.
\newblock {\em Journal of Symbolic Logic}, 61(1):52--69, 1996.

\bibitem{weier06}
A.~Weiermann.
\newblock {Classifying the Provably Total Functions of PA}.
\newblock {\em Bulletin of Symbolic Logic}, 12(2):177--190, 2006.

\bibitem{weier_draft}
A.~Weiermann.
\newblock {A Quick Proof-theoretic Analysis of $ID_1$}.
\newblock 2011.
\newblock Draft, 7 pages.

\end{thebibliography}


\end{document}